\documentclass[12pt,twoside,reqno]{amsart}
\usepackage{amsfonts}
\usepackage{amsmath,amscd}
\usepackage{amsthm}
\usepackage{latexsym}
\usepackage{amssymb}
\usepackage{enumerate}
\usepackage{url}
\newtheorem*{theorem*}{Main Theorem}
\newtheorem{theorem}{Theorem}
\newtheorem*{BasicIneq}{Basic Inequality}

\newtheorem*{rvconjecture*} {RV Conjecture}
\newtheorem{corollary}[theorem]{Corollary}
\newtheorem*{corollary*}{Corollary}

\newtheorem{lemma}[theorem]{Lemma}

\newtheorem*{claim*}{Claim}

\newtheorem*{conjecture*}{Conjecture}
\newcommand{\varpsilon}{\varepsilon}
\newcommand{\Z}{\mathbb{Z}}
\newcommand{\Q}{\mathbb{Q}}
\newcommand{\R}{\mathbb{R}}
\newcommand{\C}{\mathbb{C}}
\newcommand{\N}{\mathbb{N}}

\newcommand{\OL}{{\mathcal{O}_L}}
\newcommand{\E}{{E_\R}}

\newcommand{\D}{{\mathcal{D}}}
\newcommand{\dD}{{\mathfrak{D}}}
\newcommand{\LOG}{{\mathrm{LOG}}}
\newcommand{\Log}{{\mathrm{Log}}}
\newcommand{\AL}{{\mathcal{A}_L}}
\newcommand{\AK}{{\mathcal{A}_K}}

\newcommand{\AKk}{{\mathcal{A}^{[k]}_K}}
\newcommand{\ALj}{{\mathcal{A}^{[j]}_L}}
\newcommand{\OLU}{{\mathcal{O}_L^*}}
\newcommand{\OKU}{{\mathcal{O}_K^*}}
\newcommand{\sizeofEtor}{\lvert E_{\mathrm{tor}} \rvert}
\def\G{\R_+^{\AL}}

\def\a{\mathfrak{a}}
\newcommand{\re}{\mathrm{Re}}
\newcommand{\im}{\mathrm{Im}}

\newcommand{\QT}{Q_{\phantom{l}}^\intercal}

\newcommand{\e}{\mathrm{e}} 

\newcommand{\ie}{i.e., }

\newcommand{\usePsi}[1]{} 

\renewcommand{\digamma}{\Psi}

\newcommand{\half}{\tfrac{1}{2}}

\newcommand{\rectangle}{B}
\newcommand{\Gcal}{\mathcal{G}}
\newcommand{\GT}{\Gcal(T)}
\newcommand{\M}{\mathcal{H}}

\newcommand{\A}{{\mathcal{A}}}
\newcommand{\Vol}{\operatorname{vol}}
\newcommand{\reg}{\operatorname{reg}}
\newcommand{\Li}{\operatorname{Li}}
 \newcommand{\B}{{\mathcal{B}}}

\newcommand{\relativeunits}[2]{E(#1/#2)}
\newcommand{\ELK}{\relativeunits{L}{K}}

\DeclareMathOperator{\Span}{span}

\markleft{\hfill T. Chinburg\hfill E. Friedman\hfill J. Sundstrom}

\begin{document}

\title[A case of the Rodriguez-Villegas conjecture]
      {A case of the Rodriguez Villegas conjecture \\ with an appendix by Fernando  Rodriguez Villegas}

\author{Ted Chinburg, Eduardo Friedman and James Sundstrom}

\address{Department of Mathematics, University of Pennsylvania,
David Rittenhouse Lab.,
209 South 33rd Street, 
Philadelphia  PA 19104-6395, USA}\email{ted@math.upenn.edu}

\address{Departamento de Matem\'aticas, Facultad de Ciencias, Universidad de Chile, Las Palmeras 3425, \~Nu\~noa, Santiago  R.M., CHILE} \email{friedman@uchile.cl}
\address{The  Abdus  Salam  International  Centre  for  Theoretical  Physics,   ICTP Math Section, 
Strada  Costiera   11,­  I-­34151  Trieste,  Italy} \email{villegas@ictp.it}

\address{Department of Mathematics (038-16), 
Temple University, 
Wachman Hall, 
1805 North Broad Street, 
Philadelphia  PA 19122, 
USA} \email{james.sundstrom@temple.edu}

\thanks{Partially supported by   U.S.\ N.S.F.\ grant NSF FRG Grant DMS-1360767 (Chinburg and Sundstrom), U.S.\ 
N.S.F.\ SaTC Grants  CNS-1513671/1701785 (Chinburg) and  by  Chilean FONDECYT grant 1170176 (Friedman).}

\begin{abstract} 
 Let $L$ be a number field and let   $E $ be any subgroup of the units  $\OLU$ of $L$.  If $\mathrm{rank}_\Z(E)=1$, Lehmer's conjecture predicts that the height of any non-torsion element of $ E $ is bounded below by an absolute positive constant. If $\mathrm{rank}_\Z(E)=\mathrm{rank}_\Z(\OLU$), Zimmert proved a  lower bound on the regulator of $ E $ which grows exponentially with $[L:\Q]$. 
  Fernando Rodriguez Villegas  made a   conjecture  in 2002   that ``interpolates" between these two extremes of rank. Here we prove a high-rank case of this conjecture. Namely, it holds if $L$ contains a subfield $K$ for which $[L:K]\gg [K:\Q]$ and  $E$ contains the kernel of the norm map from $\OLU $ to $\OKU$. 
\end{abstract}
 \maketitle


\section{Introduction}\label{Introduction}
In 2002  Fernando Rodriguez Villegas  conjectured a surprising lower bound on a natural   $1$-norm  of any non-trivial  element of the $j$-th exterior power of the units of a number field. For $j$ minimal, \ie $j=1$,  Rodriguez Villegas'  conjecture
is equivalent to  Lehmer's 1933  conjectural lower bound on the height of an algebraic number \cite{Le}  \cite{Sm2}. For $j$ maximal, \ie $j=\mathrm{rank}_\Z(\OLU)$, it is equivalent to Zimmert's 1981  theorem stating that the  regulator  of a number field   grows at least exponentially with the degree of the number field \cite{Zi}. 

We   now state  his  conjecture in its strongest possible form.\footnote{\ The original 2002 write-up of this conjecture was kindly supplied to us  by F. Rodriguez Villegas and appears with his permission for the first time in print here (see \S\ref{RVappendix}). The 2002 conjecture is somewhat weaker, but F. Rodriguez Villegas  later strengthened it to the form given here.} 
\begin{rvconjecture*} (Rodriguez Villegas) There exist two absolute constants $c_0>0$ and $c_1>1$ such that for any number field $L$ and any $j\in\N$, 
\begin{equation}\label{RV}
\|\omega\|_1\ge c_0 c_1^j\qquad\qquad\Big(\forall\omega\in{\bigwedge}^j\LOG(\OLU)\subset{\bigwedge}^j\R^\AL,\ \,\omega\not=0\Big).
\end{equation}
\end{rvconjecture*}
\noindent Here ${\bigwedge}^j \LOG(\OLU)$ denotes the $j^{\mathrm{th}}$ exterior power of the lattice $\LOG(\OLU)\subset\R^\AL$,  $\AL$ denotes the set of archimedean places of $L$,  and $\LOG\colon \OLU\to\R^{\AL}$ 
is defined by 
 \begin{equation}  \label{LOG}
\big(\LOG(\gamma)\big)_v:=e_v\log|\gamma|_v,\quad e_v:=
\begin{cases}
1\ &\text{if}\ v\ \text{is real},\\
2\ &\text{if}\ v\ \text{is complex}\end{cases}
\quad \quad(\gamma\in \OLU,\ v\in\AL),
\end{equation}
where $|\ |_v$ is the  absolute value associated to $v\in\AL$ extending the usual absolute value on $\Q$. To define the 
$1$-norm in  \eqref{RV}, we start with the usual orthonormal basis $\{\delta^v\}_{v\in\AL}$ on $\R^\AL$, \ie for $w\in\AL $   
 \begin{equation}  \label{Deltav}
\delta^v_w:=
\begin{cases}
1\ &\text{if}\ \ w=v,\\
0\ &\text{if}\  \ w\not=v.
\end{cases}
\end{equation}
This gives rise to the  orthonormal basis   $\{ \delta^I\}_{I\in\ALj}$ of  
 ${\bigwedge}^j\R^\AL$, where $\ALj$ denotes the set of subsets $I$ of $\AL$ having cardinality  $j$, for each
 such $I$ we fix an ordering $\{v_1,...,v_j\}$ of $I$  and  
$$ 
\delta^I := \delta^{v_1}\wedge \delta^{v_2}\wedge\cdots \wedge  \delta^{v_j} .
$$
The $1$-norm on  ${\bigwedge}^j\R^\AL$ in the RV conjecture  \eqref{RV} is defined with respect to this  basis.  
Namely,\footnote{\ Although Rodrigez Villegas phrased the 1-norm in terms of the  archimedean  embeddings rather than    places   (see \S\ref{L1RV}), the 1-norm  is unchanged  as we inserted a  factor of 2 at complex places in  \eqref{LOG}. However,  the   embedding using places gives a larger  
2-norm if the field is not totally real, and so is better for our purposes.}  for $\omega= \sum_{I\in\ALj}c_I \delta^I$, we let $\|\omega\|_1:=\sum_{I\in\ALj}|c_I|$. 

It is worth mentioning that Siegel \cite{Sie} showed that the conjectural inequality \eqref{RV}  is not  possible for the Euclidean norm $\|\omega\|_2:=\sqrt{\sum_I c_I^2}$. Indeed, if $p>2$ is a prime, if $\varepsilon\in\C$ satisfies  $\varepsilon^p-\varepsilon+1=0$ and $L:=\Q(\varepsilon)$, then   $\|\LOG(\varepsilon) \|_2\le  \sqrt{2}\log(p) /\sqrt{p} $. Hence, the RV  conjecture 
is necessarily for the $1$-norm, at least for $j=1$. 

However, for  $j$ close to the maximal value $r_L=\mathrm{rank}_\Z(\OLU)$,  the $1$-norm and the Euclidean norm are interchangeable for the purposes of  Rodriguez Villegas' conjecture. This is simply because on any  Euclidean space $V$, we have 
$ 
  \sqrt{\dim(V)}\,\|v \|_2 \ge \|v \|_1 \ge \|v \|_2 $,
provided the $1$-norm   is taken with respect  to an orthonormal basis for $V$. 
  In this paper we will work only with the  Euclidean norm and  $j$ close to  $r_L$.

Aside from Zimmert's theorem on the regulator \cite{Zi}  and the known cases of Lehmer's conjecture   \cite{Sm2},  the cleanest result in favor of the RV conjecture is 
\begin{equation}\label{totreal}
\|\LOG(\varepsilon_1)\wedge\cdots\wedge \LOG(\varepsilon_j)\|_1 > 0.001\cdot1.4^j,
 \end{equation}
proved for all $j$, but only for totally real fields $L$.  This   follows  from work of Pohst \cite{Po} dating back to 1978. Indeed,  Pohst showed  for $L$ totally real that
$$
\|\LOG(\varepsilon)\|_2\ge\sqrt{[L:\Q]}\log\!\big((1+\sqrt{5})/2\big)\qquad\qquad( \varepsilon\in  \OLU, \ \varepsilon\not=\pm1).
$$
Using  estimates of Hermite's constant, he  deduced good lower bounds for  the regulator of a totally real field.  The same calculations show   that the  $j$-dimensional co-volume $\mu $  of the lattice spanned by  $\LOG(\varpsilon_1),...,\LOG(\varpsilon_j)$  satisfies \cite[p.\ 293]{CF}
\begin{equation}\label{totreal2}
\mu >\frac{ ([L:\Q]/j)^{j/2}1.406^j}{ (j+2)\sqrt{j}}\qquad\qquad\qquad\qquad(1\le j<[L:\Q]).
 \end{equation}
Since 
$$
  \|\LOG(\varepsilon_1)\wedge\cdots\wedge \LOG(\varepsilon_j)\|_1\ge \|\LOG(\varepsilon_1)\wedge\cdots\wedge \LOG(\varepsilon_j)\|_2=\mu,
$$ a short  numerical  computation with \eqref{totreal2} yields \eqref{totreal}.

As far as we know, the only proved cases of the RV conjecture   involve ``pure wedges," \ie $\omega$ of the form $\omega=\LOG(\varepsilon_1)\wedge\cdots\wedge \LOG(\varepsilon_j)$, where the $\varepsilon_i$ are independent elements of $\OLU$. 
  If $j = r_L$ or $j=1$, every element of $\bigwedge^j$ is (trivially) a pure wedge, but this also holds if $j = r_L-1$ (see  Lemma \ref{s:Codim1} below).  In particular, if $L$ is a totally real field of degree $n$ over $\Q$, then 
\begin{equation}\label{totrealspecial}
\|\omega\|_1 > 0.001\cdot1.4^{n-2},
 \end{equation}
for all $\omega\in{\bigwedge}^{n-2} \LOG(\OLU)$. In
general, however, the RV conjecture makes a stronger prediction than simply a lower bound on the 1-norm of pure wedges.

  Another known case of the RV conjecture occurs when   
\begin{equation}\label{RelUnits}
E=E(L/K):=\big\{\varepsilon\in\OLU\big|\,\text{Norm}_{L/K}(\varepsilon)\text{ is a root of unity} \big\} 
\end{equation}
is the group  of relative units associated to an extension $L/K$. 
    Friedman and Skoruppa \cite{FS} proved in 1999 that  inequality \eqref{RV}  in the RV  conjecture holds for pure wedges if $[L:K]\ge N_0$  for some absolute constant  $N_0$.\footnote{\ The inequality proved in \cite{FS} is for the relative regulator $\text{Reg}(L/K)$ rather than for the   co-volume $\mu$ of the relative units. This suffices since $\mu=\text{Reg}(L/K)\prod_{v\in\AK}  \sqrt{r_v}\ge\text{Reg}(L/K)$, where $r_v$ is the number of places of $L$ above $v$. The proof of this relation between the co-volume and the relative regulator mimics the  determinant manipulations in the case $K=\Q$ \cite[p.\ 115]{BS}. We note that J. Sundstrom, in the appendix to his doctoral thesis \cite{Su1}, corrected an error in  Skoruppa and Friedman's proof.  Namely, in the bound on what is called $J_1$ in the proof of Lemma 5.5 of \cite{FS}, the real part of the error term $\rho$ in the exponential was neglected. This did  not affect the proof of their Main Theorem, but it did affect the numerical constants claimed in Theorem 4.1 and its corollaries. By improving the asymptotic estimates in \cite{FS} and using extensive computer calculations, Sundstrom was able to prove the estimate in Theorem 4.1 of \cite{FS}, with the constants as given there, In particular,  $N_0=40$. If we are willing to settle for $N_0=400$, the proof in \cite{FS} will do after adjusting the constants to correct for the error in the proof of Lemma 5.5.} To prove their result, Friedman and Skoruppa defined a $\Theta$-type series $\Theta_E$ associated to any subgroup $E\subset\OLU$ of arbitrary rank and used it to produce a complicated inequality for the co-volume $\mu(E)$ associated to the lattice $\LOG(E)$. In the case of $E=E(L/K)$ they obtained the desired inequality using the saddle-point method to estimate     the terms in the series $\Theta_E$ as $[L:K]\to\infty$.  Although the saddle-point method in one variable is a standard tool, the   difficulty in the  asymptotic estimates in \cite[\S5]{FS} was that the estimates needed  to depend only on $[L:K]$.

The results cited so far all pre-date the RV conjecture and essentially dealt with regulators or Lehmer's conjecture. Inspired by the RV conjecture, Sundstrom \cite{Su1} \cite{Su2} dealt in his 2016 thesis with a new kind of subgroup of the units. Namely, suppose $L$ contains two distinct real quadratic subfields $K_1, K_2$, and let
$
E := E(L/K_1) \cap E(L/K_2).
$
The series $\Theta_E$ is still defined and yields an inequality for the co-volume $\mu\big(\LOG(E)\big)$, but to estimate the  terms in the inequality Sundstrom had to apply the saddle-point method  to a triple integral. Keeping all estimates uniform in this case proved considerably harder than in the one-variable case  treated in \cite{FS}. In the end, Sundstrom was able to verify the RV conjecture in this case for pure wedges. More precisely, he proved the existence of absolute constants $N_0,\,c_0>0$ and $c_1>1$ such that  $\mu(E)\ge c_0 c_1^j$ for $[L:\Q]\ge N_0$ and  $j=\text{rank}_\Z(E)=\text{rank}_\Z(\OLU)-2$.

Here we extend Sundstrom's result,   letting the $K_i$ be arbitrary, as follows. 
\vskip.3cm
 \noindent{\it{Let $K_1,\ldots, K_\ell$ be subfields of a number field $L$,  let   $K:=K_1\cdots K_\ell$ $\subset L$ be  the compositum of the $K_i$,   let $E :=\bigcap_{i=1}^\ell E(L/K_i)\subset\OLU$ be  the subgroup of the units  of $L$ whose norm to each $K_i$ is a root of unity, and let $\varpsilon_1,...,\varepsilon_j$ be independent elements of $E$, where $j:=\mathrm{rank}_\Z(E)$. Then there is an absolute constant $N_0$ such that 
$$
\| \varpsilon_1\wedge\cdots\wedge\varepsilon_j\|_1\ge\|\varpsilon_1\wedge\cdots\wedge\varepsilon_j\|_2\ge   1.1^j,
$$ 
whenever  $[L:K]\ge N_0  \cdot  2.01^{ [K:\Q]} $. }}

\vskip.3cm

In fact the above is an immediate corollary of  our 
\begin{theorem*}  Suppose $E \subset\OLU$ is a subgroup of the units of the number field $L$ such that  $E(L/K)\subset E$ for some subfield $K\subset L$, where $E(L/K)$ are the relative units   defined in \eqref{RelUnits}. Let $\varepsilon_1,...,\varepsilon_j$ be independent elements of $E$, where $j:=\mathrm{rank}_\Z(E)$. Then the RV conjecture \eqref{RV} holds for
$\omega:=\varpsilon_1\wedge\cdots\wedge\varepsilon_j$ and $[L:K] $ large enough compared to $[K:\Q]$. 

More precisely, there is an absolute constant  $N_0 $  such that if $[L:K]\ge N_0\cdot 2.01^{ [K:\Q]} $, then
\begin{equation}\label{Expgrowth}
\| \varpsilon_1\wedge\cdots\wedge\varepsilon_j\|_1\ge\|\varpsilon_1\wedge\cdots\wedge\varepsilon_j\|_2\ge  1.1^j\qquad\qquad\big(j:=\mathrm{rank}_\Z(E)\big).
\end{equation}
 \end{theorem*}

\noindent Our proof of the Main Theorem is again through an asymptotic analysis of the inequality for $\Theta_E$ in \cite{FS}, but there are several new features which bring the proof closer to the case of a general high-rank subgroup $E\subset\OLU$. 

In both \cite{FS} and \cite{Su2}, the uniformity of the asymptotic estimates depends  on having explicit expressions for the orthogonal complement of $\LOG(E)$ inside $\R^\AL$, but here  we  have very little knowledge of  $\LOG(E)^\perp$. 
 As in \cite{FS} and \cite{Su2}, we take a  Mellin transform of the terms of $\Theta_E$ and invert it to express each term in $\Theta_E$ as a $k$-dimensional complex contour integral (see Lemma \ref{psiinversemellin} below). Here $k:= 1+\text{rank}_\Z(\OLU/E)$ is     the co-rank of $E\subset\OLU$, shifted by 1. 
  
To apply the saddle-point method to our  integral, we need a saddle point. In the case of \cite{FS}  one could easily write down a formula for the saddle point in terms of the logarithmic derivative of the classical $\Gamma$-function. In \cite{Su2}  the equations for the critical point were explicit enough that monotonicity arguments proved the  existence  of the saddle point. In our case the equations are too complicated to analyse directly. Instead, in \S\ref{CriticalPoint} we obtain the existence and uniqueness of the saddle point by  re-interpreting it as the value of  the Legendre transform of a convex function on $\R^k$, closely related to $\log\Gamma$. 

Since (what will prove to be) the main term in our asymptotic expansion depends on the saddle point $\sigma=(\sigma_1,...,\sigma_k)\in\R^k$, of which we can only  control $\sigma_1$, in \S\ref{IneqCrit} we prove inequalities for  the main term which  depend only on $\sigma_1$. 
We need these inequalities   to prove that the main term has the exponential growth claimed in the Main Theorem.

The results proved in \S\ref{FundIneq}-\S\ref{IneqCrit}   are valid for any subgroup $E\subset\OLU$. In  \S\ref{Asymptotics} we carry out the required  uniform asymptotic estimates, assuming  $E(L/K)\subset E$ and $[L:K]\gg0$ to show that the purported main term actually dominates. Finally, in \S\ref{ProofMainTheorem} we put everything together and prove the Main Theorem.

\section{The $\Theta$-function}\label{FundIneq}
In this section   we recall   the  series  $\Theta_E(t;\a)$ associated to a  subgroup $E\subset\OLU$ of the units and to a fractional  ideal $\a$ of the number field $L$. We also recall the inequality for the co-volume of $\LOG (E)$ resulting from the functional equation of $\Theta_E$. This is all quoted from  \cite[\S2]{FS}. Our main new task here is to express the terms in the inequality as an inverse Mellin transform.

\subsection{The basic inequality}\label{ThetaSetup}
 Given a subgroup $E\subset\OLU$, we define $\E\subset\G$ as the group  generated by all elements of the form
\begin{equation*}
 x=(x_v)_{v\in\AL}=\big(\,|\varepsilon|_{v^{{\phantom{-1}}}}^{\xi_{\phantom{-1}}}\!\!\!\!\big)_{v\in\AL}\qquad\qquad\big(\varepsilon\in E,\ \, \xi\in\R\big).
\end{equation*}
Here  $\R_+:= (0,\infty)$ is the multiplicative group of the positive real numbers, $\AL$ denotes the set of Archimedean places of $L$, and $|\ |_v$ is the (un-normalized) absolute value  associated to the archimedean place $v\in\AL$. Thus, for $a\in L$ we have
\begin{equation}\label{Norm}
 |\mathrm{Norm}_{L/\Q}(a)|=\prod_{v\in\AL}|a|_v^{e_v},\qquad(e_v:=1\ \text{if} \ v\ \text{is real},\ e_v:=2\ \text{if} \ v\ \text{is complex}).
 \end{equation}
 Note that
\begin{equation}\label{sumev}
\sum_{v\in\AL} e_v = [L:\Q] =: n,
\end{equation}
\begin{equation}\label{Units}
\prod_{v\in\AL}x_v^{e_v}=1\qquad\qquad \big(x=(x_v)_{v\in\AL}\in \E \big),
\end{equation}
and that $\varepsilon\in E$ acts on $x=(x_v)_v\in \E$, via $(\varepsilon\cdot x)_v := |\varepsilon|_v\,x_v$.

We fix a Haar measure   on $\E\subset\R_+^\AL$ as follows. The standard Euclidean structure on $\R^\AL$, in which the $\delta^v$ in \eqref{Deltav} form an orthonormal basis of $\R^\AL$, induces a Euclidean structure (and therefore a unique  Haar measure)  on any $\R$-subspace of $\R^\AL$. We give $\E$ the Haar measure $\mu_\E$ that results from pulling back the Haar measure on the  $\R$-subspace $\LOG(\E)$ via the isomorphism $\LOG$, and let $\mu_\E(\E/E)$ be the measure of a fundamental domain for the action of $E$ on $\E$.
 
 Following \cite[p.\ 120]{FS}, for a fractional ideal $\a\subset L$ and   $t>0$, we  let
\begin{equation}\label{ThetaE}
 \Theta_E(t;\a):= \frac{\mu_\E(\E/E)}{\sizeofEtor}+\sum_{\substack{a\in \a/E\\ a\not=0}}\int_{x\in\E}\e^{-c_\a t\,\|ax\|^2}\,d\mu_\E(x),\quad\ \|ax\|^2:=\sum_{v\in\AL} e_v|a|_v^2 x_v^2,
\end{equation}
where   $\sizeofEtor$ is the number of roots of unity in $E$,
$$
c_\a := \pi\big(\sqrt{|D_L|}\,\mathrm{Norm}_{L/\Q}(\a)\big)^{-2/n}, \qquad \quad D_L:=\text{discriminant of }L,\quad  n:=[L:\Q] .
$$
  Note that the integral in \eqref{ThetaE} depends only on the $E$-orbit of $a$, and hence is independent of the representative $a\in\a/E$ taken for the $E$-orbit of $a$.

Our starting point for proving lower bounds on co-volumes is the inequality  \cite[ Corol.\ p.\ 121]{FS}, valid for any $t>0$ and any fractional ideal $\a$ of $L$.
\begin{equation}\label{ThetaIneq}
\Theta_E(t;\a)+\frac{2t \Theta_E^\prime(t;\a)}{n}\ge0\qquad\qquad\qquad\Big(t>0,\ \, \Theta_E^\prime:=\frac{d\Theta_E}{dt}\Big).
\end{equation}
Writing out the individual terms of \eqref{ThetaIneq}, we have  \cite[p.\ 121, eq.\ (2.6)]{FS} the
\begin{BasicIneq}
\begin{equation}\label{MainIneq}
\frac{\mu_\E(\E/E)}{\sizeofEtor}\ge \sum_{\substack{a\in \a/E\\ a\not=0}}\int_{x\in\E}\Big(\frac{2t\|ax\|^2}{n}-1\Big)\e^{- t\,\|ax\|^2}\,d\mu_\E(x)\quad\quad(t>0).
\end{equation}
\end{BasicIneq}
\noindent Note that in  \cite{FS} we find $ tc_\a$ instead of $t$ in \eqref{MainIneq}, but $t>0$ is arbitrary there too.
\subsection{Mellin transforms}
Our main task in this section is to re-write the $r$-dimen\-sion\-al  integral in \eqref{ThetaE} as an inverse Mellin transform.
 For this it will prove convenient to characterize $\E\subset G:=\R_+^\AL$ not through  generators, but rather through generators of the orthogonal complement in $\R^\AL$ of $\Log_G(\E)$. Here $ \Log_G\colon G\to\R^\AL$ is the group isomorphism defined by
\begin{equation}\label{Log}
\big(\Log_G(g)\big)_v:=\log(g_v)\qquad\qquad\big(v\in\AL,\ \, g=(g_v)_v\in G :=\R_+^\AL\big).
\end{equation}
Note that $\Log_G$ is {\it{not}} the traditional logarithmic embedding $\LOG$ in \eqref{LOG}, as we do not insert a factor of $e_v$ in \eqref{Log}. Instead we  endow $\R^\AL$ with a new inner product
\begin{equation}\label{InnerProd}
\langle \beta,\gamma\rangle:=\sum_{v\in\AL}e_v\beta_v\gamma_v\qquad\qquad\big(\beta=(\beta_v)_v,\ \gamma=(\gamma_v)_v\in\R^\AL\big),
\end{equation}
where $e_v=1$ or 2 as in \eqref{Norm}.
Let  $ \big\{q_j\big\}_{j=1}^k=\big\{(q_{jv})_v\big\}_{j=1}^k$
be an   $\R$-basis of the orthogonal complement of $\Log_G(\E)$ in $\R^\AL$   such that   
\begin{equation}\label{qjs}
 q_{1v}:=1\ \ (\forall v\in\AL),  \ \ \sum_{v\in\AL}e_vq_{iv}q_{jv}=0\ \ \big( 1\le i\not= j\le   k:=1+\mathrm{rank}_\Z(\OLU/E)\big).
\end{equation}
 Thus, for  $g=(g_v)_v\in G:=\R_+^\AL$,
\begin{equation}\label{CharE}
g\in\E\ \qquad \Longleftrightarrow \qquad \  \sum_{v\in\AL}e_v q_{jv}\log(g_v)=0\qquad(1\le j\le k).
\end{equation}

Let $H:=\R_+^k$. Define a homomorphism $\delta\colon G\to H$ by
\begin{equation}\label{delta}
\big(\delta(g)\big)_j:=\prod_{v\in\AL} g_v^{e_vq_{jv}}\qquad\qquad\big(1\le j\le k,\ \,g=(g_v)_v\in G:=\R_+^\AL),
\end{equation}
so that by \eqref{CharE} we have an exact sequence
\begin{equation}\label{exact}
\begin{CD}
1  @>  >> \E@>  >> G @>\delta >> H@>  >> 1.
\end{CD}
\end{equation}
Let $\sigma\colon H\to G$ be a homomorphism splitting the exact sequence \eqref{exact}, \ie $\delta\circ\sigma$ is the identity map on $H$. Such a splitting exists because
 $G$ and $H$ are real
 vector spaces.
Let
\begin{equation}\label{muG}
d\mu_G:=\prod_{v\in\AL}\frac{dg_v}{g_v},\qquad \qquad \qquad d\mu_H:=\prod_{j=1}^k\frac{dh_j}{h_j}
\end{equation}
be the usual Haar measures on $G:=\R_+^\AL$ and  $H:=\R_+^k$.

Recall that in order to define $\Theta_E$   in \eqref{ThetaE}  we   fixed a Haar measure $\mu_\E$ on $\E$. In order to calculate   Mellin transforms    below, we will need to compare the Haar measure  $\mu_H\times\mu_\E$ on $H\times\E$ with a Haar measure coming from
 $\mu_G$. Namely, if $\gamma\colon \E\times H\to G$ is the isomorphism defined by the splitting $\sigma$, \ie
 \begin{equation}\label{gamma}\gamma(x,h):=x\sigma(h),
 \end{equation}
then the measure 
$\mu_G \circ \gamma$ is a Haar measure on  $\E \times H$. 
Hence
 \begin{equation}\label{c}
 c \,\mu_G\circ\gamma=\mu_\E\times  \mu_H,
\end{equation}
where the positive constant $c$ is evaluated in the next lemma.

 \begin{lemma}  \label{ccomputation}
Let $Q$ be the $|\AL|\times k$ matrix whose rows are indexed by $v \in \AL$ and whose columns are indexed by $j = 1,\ldots,k$, with entry $Q_{v,j} := q_{jv}$ in the $v^\text{th}$ row and the $j^\text{th}$ column, with $q_{jv}$ as in \eqref{qjs}.
Then $c$ in \eqref{c} is independent of the splitting $\sigma$ in \eqref{gamma} and is given by 
\begin{equation}
\label{eq:canswer}
c  =   2^{r_2}   \sqrt{ \mathrm{det}(\QT Q) },
\end{equation}
where $\QT$ is the transpose of $Q$ and $r_2$ is the number of complex places of $L$.
\end{lemma}
\begin{proof}  For $x=(x_v)$ and $y=(y_v)\in\R^\AL $, let $x\cdot y$ be the standard dot product  $x\cdot y:=\sum_{v\in\AL} x_v y_v$.   
Recall that we   defined in \eqref{InnerProd} another inner product on $\R^\AL$, namely  $\langle x,y\rangle:=\sum_v e_v x_v y_v$. To relate these products, let $T:\R^\AL \to\R^\AL $ be given  by $\big(T(x)\big)_v:=e_vx_v$. Then  
 \begin{equation}\label{TT}
\langle x,y\rangle=x\cdot T(y)=T(x)\cdot y. 
\end{equation} 
Note that $\det(T)=2^{r_2}$.


Let $u_1,...,u_r$ be an  orthonormal basis of $V$ (with respect to the dot product),  let $C_1:=\big\{ \sum_\ell x_\ell u_\ell\big|\,0\le x_\ell\le1\big\}\subset V$ be the $r$-cube   spanned by the $u_\ell$, and let $B_1:=\LOG^{-1}(C_1)$. By the definition of the measure $\mu_\E $ given in the paragraph preceding \eqref{ThetaE}, $\mu_\E(B_1)=1$.


We define next  an analogous subset $B_2\subset H:=\R_+^k$ with $\mu_H(B_2)=1$. Let $F_1, \dotsc, F_k$ be the ``standard'' orthonormal  basis of $\R_+^k$ as an $\R$-vector space; that is, $(F_j)_i = \e$ if $i=j$, and $(F_j)_i = 1$ otherwise. Let 
$B_2\subset \R_+^k$ be the  $k$-cube  spanned by $F_1, \dotsc, F_k$, so that   $\mu_H(B_2) = 1$.

Set $B:=B_1\times B_2\subset\E\times H$, so that $(\mu_\E\times\mu_H)(B)=1$. Thus   $c$ in \eqref{c} satisfies
\begin{equation}\label{cB}
 c^{-1}=\mu_G\big(\gamma(B)\big).
\end{equation}
Now, $\gamma(x,h):=x\sigma(h)$ and $\mu_G$  is the measure on $G$ that maps by $\Log_G$  to the standard Haar measure 
on $\R^\AL$  $\big($see  \eqref{Log},    \eqref{muG} and \eqref{gamma}$\big)$. Hence,  $c^{-1}=|\!\det(M)|$,
where $M$ is the $(|\AL|\times|\AL|)$-matrix whose first $r$ columns are the vectors $w_\ell:=\Log_G\big(\LOG^{-1}(u_\ell)\big)\in\R^{\AL}\ \,(1\le\ell\le r )$. 
The remaining $k$ columns of $M$ are the vectors $\Log_G\big(\sigma(F_j)\big)\ \,(1\le j\le k )$.

Suppose $\tilde\sigma$ is another splitting of \eqref{exact}. Then $\sigma(F_j)\tilde\sigma(F_j)^{-1}\in\E$, and therefore 
$\Log_G\big(\sigma(F_j)\big)-\Log_G\big(\tilde\sigma(F_j)\big)$ lies in the span of the columns $w_1,...,w_r$. Hence $c$ is independent of the splitting $\sigma$, as claimed in the lemma. We are therefore free  to use the splitting $\sigma$ determined by
\begin{equation*}\label{sigma}
\big(\sigma(F_j)\big)_v:=\exp(q_{jv}/d_j)\quad\quad\ \Big(v\in\AL,\ 1\le j\le k,\ \,d_j:= \langle q_j,q_j \rangle:=\sum_{\rho\in\AL}e_\rho q_{j\rho}^2\Big).
\end{equation*}
Using   \eqref{delta}  and the orthogonality relations 
\eqref{qjs}, one checks that this is  indeed a splitting of $\delta$. With this $\sigma$,  the last $k$ columns of $M$ are just  $\Log_G\big(\sigma(F_j)\big)=d_j^{-1}q_j\in\R^\AL$. As $T\circ\Log_G=\LOG$ and $\det(T)=2^{r_2}$ $\big($see 
\eqref{TT}$\big)$, we have 
$$
c^{-1}=|\!\det(M)|=2^{-r_2}|\!\det(N)|,$$
where $N$ is the $(|\AL|\times|\AL|)$-matrix whose columns are $T$ applied to the columns of $M$, \ie the   columns of $N$ are $u_1,...,u_r$, followed by
  $d_1^{-1}T(q_1),...,d_k^{-1}T(q_k)$. 

To prove the lemma we must show that $|\!\det(N)|^{-1}=\sqrt{ \mathrm{det}(\QT Q) }$. We calculate   
$ |\!\det(N)|$ as 
$$|\!\det(N)|= \frac{|\!\det(R_{\phantom{l}}^\intercal N)|}{\sqrt{\det(R_{\phantom{l}}^\intercal  R)}},
$$ 
where $R$ is the $(|\AL|\times|\AL|)$-matrix whose columns   are $u_1,...,u_r$, followed by
 $q_1,...,q_k$ (\ie   $Q$). 
Using the orthonormality  of the $u_\ell$'s (with respect to the dot product),  we see that 
$R_{\phantom{l}}^\intercal  R$ can be divided into four blocks,   the upper left one being  the $r\times r$  identity matrix $I_{r\times r}$.
Below it, $R_{\phantom{l}}^\intercal  R$ has a $k\times r$ block with entries 
$$q_j\cdot u_\ell = q_j\cdot T\big(\Log_G(\LOG^{-1}(u_\ell))\big)  =\big\langle q_j,\Log_G\big(\LOG^{-1}(u_\ell)\big)\big\rangle=0,$$
where we used    \eqref{TT} and the   definition of  the $q_j$'s as a basis of the orthogonal complement of $\Log_G(E_\R)\subset\R^\AL$  $\big($with respect to $\langle  \ \rangle $, see \eqref{CharE}$\big)$. Since the bottom right $k\times k$ block of $R_{\phantom{l}}^\intercal  R$ is $\QT Q$, we find that 
$R_{\phantom{l}}^\intercal  R=\begin{pmatrix} I_{r\times r}& 0_{r\times k} \\ 
                0_{k\times r}    & \QT Q\end{pmatrix} $. Thus, 
$\det\!\big(R_{\phantom{l}}^\intercal  R\big)=\sqrt{ \mathrm{det}(\QT Q) }$. A similar calculation shows   $ R_{\phantom{l}}^\intercal  N=\begin{pmatrix} I_{r\times r}& *_{r\times k} \\ 
                0_{k\times r}    &I_{k\times k}\end{pmatrix} $, whence $\det(R_{\phantom{l}}^\intercal  N)=1$.\end{proof}

In order to study the $\Theta$-series \eqref{ThetaE}, we need to consider integrals of the form
\begin{equation}\label{BigPsi}
\int_{x\in\E}\e^{- \|gx\|^2}\,d\mu_\E(x) \qquad\qquad\qquad\big( \|gx\|^2 := \sum_{v\in\AL}e_v g_v^2 x_v^2\big),
\end{equation}
for $g=(g_v)_v\in G:=\R_+^\AL$. For $h=(h_1, \ldots, h_k)\in H:=\R_+^k$, define $\psi$ by substituting $g=\sigma(h)$ above:
\begin{equation}\label{littlepsi}
\psi(h) := \int_{x\in\E}\e^{- \|\sigma(h)x\|^2}\,d\mu_\E(x).
\end{equation}
Note that the integral \eqref{BigPsi} depends only on $g$ modulo $\E$, so the function $\psi$ is independent of the choice of $\sigma$ splitting the exact sequence \eqref{exact}. 
The fact that \eqref{BigPsi} depends only on $g$ modulo $\E$ also shows that
\begin{equation}\label{Psipsi}
\int_{x\in\E}\e^{- \|gx\|^2}\,d\mu_\E(x) 
  =\int_{x\in\E}\e^{- \|\sigma(\delta(g))x\|^2}\,d\mu_\E(x)
	= \psi\big(\delta(g)\big),
\end{equation}
so we will concentrate on $\psi$, a function of only $k$ variables.

Define a linear map $S\colon \C^k \to \C^\AL$ by $S(s) = Qs$, where $Q$ is the matrix whose $j^\text{th}$ column is $q_j\in\R^\AL\subset\C^\AL$, as in Lemma \ref{ccomputation}. Also define maps $S_v\colon \C^k \to \C$ for each $v\in\AL$ by $S_v(s) = \big(S(s)\big)_v$. That is,
\begin{equation}
\label{S}
S(s) = \sum_{j=1}^k s_j q_j , \quad\quad\quad\quad S_v(s) = \sum_{j=1}^k q_{jv} s_j\qquad\big(s=(s_1,...,s_k)\big).
\end{equation}
Note that $S$ is injective since the $q_j\in\R^\AL$ are linearly independent.

Our first aim is to calculate the ($k$-dimensional) Mellin transform
\begin{equation}\label{Mellinpsi}
(M \psi)(s):=\int_H\psi(h)\,h^s\,d\mu_H(h):=\int_{h_1=0}^\infty \cdots\int_{h_k=0}^\infty \psi(h)\,  h_1^{s_1}\cdots h_k^{s_k}\frac{dh_1}{h_1}\cdots\frac{dh_k}{h_k},
\end{equation}
where $\re(s):=\big(\re(s_1), \ldots, \re(s_k)\big)\in\D $, with
\begin{equation}\label{D}
\D:=\Big\{\sigma=(\sigma_1, \ldots, \sigma_k)\in\R^k\big|\  S_v(\sigma)>0\ \, \forall v\in\AL\Big\}.
\end{equation}
As $q_{1v}:=1$ for all $v\in\AL\ \big($see \eqref{qjs}$\big)$, for $t>0$ we have  $(t,0,0, \ldots, 0)\in\D$. Hence $\D$ is a non-empty, open, convex subset of $\R^k$.
We will presently  prove that the Mellin transform  $(M\psi)(s)$ in \eqref{Mellinpsi} converges if $\re(s)\in\D$.

 In the following calculation of  $(M\psi)(s)$  the reader should initially consider only real $s_j$, so that the integrand is positive. At the end of the calculation it will become clear that the integral converges for $s$ in the open subset of $\C^k$
where $\re(s)\in\D$. 
\begin{align*}
(M\psi &)(s) = \int_{h\in H} \int_{x\in\E} h^s\exp(-\|x\sigma(h)\|^2)\,d\mu_\E(x) \,d\mu_H(h)  \\ 	&
= \int_{(x,h)\in \E\times H} h^s\exp(-\|x\sigma(h)\|^2)\,d(\mu_\E \times\mu_H)(x,h)   \\
          &= \ 2^{r_2}\sqrt{ \mathrm{det}(Q_{\phantom{l}}^\intercal Q) } \int_{(x,h)\in \E\times H} \big(\delta(\gamma(x,h))\big)^s\exp(-\|\gamma(x,h)\|^2)\,d(\mu_G\circ\gamma)(x,h)  ,
\end{align*}
where in the last step we used Lemma \ref{ccomputation} and
  $\delta\big( \gamma(x,h)\big)=\delta\big(\sigma(h)x\big)=h$, with  $ \delta $ as in \eqref{delta}.  
Next we  substitute $g = \gamma(x,h)$ to get 
\begin{align}
&(M\psi)(s) = 2^{r_2}\sqrt{ \mathrm{det}(Q_{\phantom{l}}^\intercal Q) }\int_{g\in G} \delta(g)^s \exp(-\|g\|^2)\,d\mu_G(g) \nonumber \\
           &= 2^{r_2}\sqrt{ \mathrm{det}(Q_{\phantom{l}}^\intercal Q) }\int_{g\in G}  \exp(-\|g\|^2) \prod_{j=1}^k\delta(g)_j^{s_j}\,d\mu_G(g)  \nonumber 
\\ 
					 &= 2^{r_2}\sqrt{ \mathrm{det}(Q_{\phantom{l}}^\intercal Q) } \int_{g\in G} \exp(-\|g\|^2)\prod_{j=1}^k\Big(\prod_{v\in\AL}g_v^{e_v q_{jv}}\Big)^{s_j}\,d\mu_G(g)      \nonumber\\   \nonumber 
					 &= 2^{r_2}\sqrt{ \mathrm{det}(Q_{\phantom{l}}^\intercal Q) } \int_{g\in G} \exp\!\Big(\!-\sum_{v\in\AL} e_v g_v^2\Big) \prod_{v\in\AL}g_v^{e_v \sum_{j=1}^k q_{jv}s_j}\ \prod_{v\in\AL}\frac{dg_v}{g_v}\\
					 &= 2^{r_2}\sqrt{ \mathrm{det}(Q_{\phantom{l}}^\intercal Q) }\prod_{v\in\AL} \int_{0}^\infty  g_v^{e_v S_v(s)}\exp(-e_v g_v^2) \frac{dg_v}{g_v}
			       = \frac{ \sqrt{ \mathrm{det}(Q_{\phantom{l}}^\intercal Q) }}{2^{r_1}} \prod_{v\in\AL} \frac{\Gamma(e_v S_v(s)/2)}{e_v^{e_v S_v(s)/2}} , 	 \label{MT}
\end{align}
where $r_1$ is the number of real places of $L$.


\begin{lemma}\label{MinVer}
For any $\sigma \in \D$  \big(see (\ref{D})\big), the Mellin inversion formula holds:
\begin{equation}\label{Minversion}
\psi(h) = \frac{1}{(2\pi i)^k}\int_{I_\sigma} (M\psi)(s) \, h^{-s}\,ds \qquad\qquad\qquad\qquad (h\in\R_+^k),
\end{equation}
where    $s=(s_1,...,s_k)$ and $ I_\sigma \subset \C^k$ is the product of the $k$ vertical lines $\re(s_j) = \sigma_j$, taken from $\sigma_j - i\infty$ to $\sigma_j + i\infty$.
\end{lemma}
\begin{proof}
The calculation \eqref{MT} shows that the Mellin transform $(M\psi)(s)$ is defined for $s\in I_\sigma$. Thus Mellin inversion will work provided that $\int_{I_\sigma} \big|(M\psi)(s)h^{-s}\,ds\big|<\infty$. Since $\big|h^{-s}\big|$ and $\big|e_v^{e_v S_v(s)/2}\big|$ are constant on $I_\sigma$, we turn to the factors $|\Gamma(e_v S_v(s)/2)|$ in \eqref{MT}. 
Write $s = \sigma + iT$, $T\in\R^k$. 
In a strip $0<C_1\le\re(z)\le C_2$, we have $|\Gamma(z)|<C_3\exp\!\big(\!-|\im(z)|\big)$.\footnote{\ In fact, $|\Gamma(z)|<C_\varepsilon\exp(-(\pi-\varepsilon)|\im(z)|/2)$ holds for any $\varpsilon>0$ \cite[Cor.\ 1.4.4]{AAR}.} Since $\re\big(e_v S_v(s)\big) = e_vS_v(\sigma)>0$   for $s\in I_\sigma$,
$$
\prod_{v\in\AL}|\Gamma(e_v S_v(s)/2)|< C_4\exp\!\Big(\!-\sum_{v\in\AL}e_v \big|S_v(T)\big|/2\Big)\le C_4\exp\!\big(\!- \|S(T)\|_1/2\big),
$$
where $\|(m_v) \|_1:=\sum_{v\in\AL}|m_v|$ is the 1-norm on $\R^\AL$, and $S$ is the linear function from \eqref{S}. Since $S$ is injective, there exists $C_5>0$ such that
$$
\|S(T)\|_1 \ge C_5\|T\|_1 := C_5\sum_{j=1}^k |T_j|.
$$
Thus $(M\psi)(s)h^{-s}$ is integrable over $I_\sigma$ and Mellin inversion \eqref{Minversion} holds.
\end{proof}

Let 
\begin{equation}\label{Gammav}
\Gamma_v(z):=
 \begin{cases}
  \Gamma(z)\ \,&\text{if}\ v\ \text{is}\ \text{real},\\
  \Gamma(z)\Gamma(z+\frac12)\ \,&\text{if}\ v\ \text{is}\ \text{complex},
 \end{cases}
\end{equation}
and 
\begin{equation}\label{alpha}
{\alpha}(s) := \sum_{v\in\AL}\log \Gamma_v\big(S_v(s)\big)
                   = \sum_{v\in\AL}\log \Gamma_v\big( {\textstyle \sum_{j=1}^k q_{jv}s_j}\big).
\end{equation}
We take the branch of  $\log\Gamma_v(z)$  which is real when $z$ is real and  positive.

\begin{lemma}\label{psiinversemellin}
Let $y=(y_1, \ldots, y_k)\in\R^k$ and $\chi=\chi(y):=(\e^{y_1/2}, \ldots, \e^{y_k/2})\in H:=\R_+^k$. Then 
\begin{equation}\label{psiIM}
\psi(\chi)= \frac{\sqrt{\det(Q_{\phantom{l}}^\intercal Q)}}{2^{r_1}(2\sqrt{\pi})^{r_2}(\pi i)^k}\int_{s\in I_\sigma}  \exp\!\big({\alpha}(s)-\textstyle\sum_{j=1}^k y_j s_j\big)  \,ds \ \qquad(\mathrm{for\ any}\ \sigma\in\D),
\end{equation}
with  $\psi$  as in \eqref{littlepsi},  ${\alpha}$ as in \eqref{alpha}, $Q$  as in Lemma $\mathrm{\ref{ccomputation}}$, $I_\sigma$ as in Lemma $\mathrm{\ref{MinVer}}$, and $r_1$ (resp.\  $r_2$) being  the number of real (resp.\ complex) places of $L$.
\end{lemma}
\begin{proof}
If $v$ is complex, so $e_v=2$, the duplication formula  gives
\begin{equation*}
\frac{\Gamma\big(e_v S_v(s)\big)}{e_v^{e_v S_v(s)}} 
         = \frac{\Gamma\big(2S_v(s)\big)}{2^{2S_v(s)}}
				 = \frac{\Gamma\big(S_v(s)\big)\Gamma\big(\tfrac12+S_v(s)\big)}{2\sqrt{\pi}}
				 = \frac{\Gamma_v\big(S_v(s)\big)}{2\sqrt{\pi}}.
\end{equation*}
If $v$ is real, so $e_v=1$, then
\begin{equation*}
\frac{\Gamma\big(e_v S_v(s)\big)}{e_v^{e_v S_v(s)}} = \Gamma\big(S_v(s)\big) = \Gamma_v\big(S_v(s)\big).
\end{equation*}
From \eqref{MT} and Mellin inversion \eqref{Minversion} we get
\begin{align*}
\psi(\chi) 
  &= \frac{1}{(2\pi i)^k}\int_{s\in I_{\sigma/2}} \chi^{-s}\cdot(M\psi)(s)\,ds \\
  &= \frac{\sqrt{\det(Q_{\phantom{l}}^\intercal Q)}}{2^{r_1}(2\pi i)^k} \int_{s\in I_{\sigma/2}} \prod_{j=1}^k \chi_j^{-s_j}\cdot \prod_{v\in\AL} \frac{\Gamma\big(\frac{e_v S_v(s)}{2}\big)}{e_v^{e_v S_v(s)/2}}\,ds \\ 
	&= \frac{\sqrt{\det(Q_{\phantom{l}}^\intercal Q)}}{2^{r_1}(\pi i)^k} \int_{s\in I_\sigma} \prod_{j=1}^k \chi_j^{-2s_j}\cdot \prod_{v\in\AL} \frac{\Gamma\big(e_v S_v(s)\big)}{e_v^{e_v S_v(s)}}\,ds \\ 
	&= \frac{\sqrt{\det(Q_{\phantom{l}}^\intercal Q)}}{2^{r_1}(2\sqrt{\pi})^{r_2}(\pi i)^k}\int_{s\in I_\sigma}  \exp\!\big(\!-{\textstyle\sum_{j=1}^k y_j s_j\big)} \prod_{v\in\AL} \Gamma_v\big(S_v(s)\big) \,ds \\ 
	&= \frac{\sqrt{\det(Q_{\phantom{l}}^\intercal Q)}}{2^{r_1}(2\sqrt{\pi})^{r_2}(\pi i)^k} \int_{s\in I_\sigma} \exp\!\big({\alpha}(s)-\textstyle\sum_{j=1}^k y_j s_j\big) \,ds. \qedhere
\end{align*}
\end{proof}

Now we apply the lemma to the Basic Inequality \eqref{MainIneq}.
\begin{corollary}\label{sigma1ineq}  For  $t>0$ and $ a\in L^*$,  define $y=y_{a,t}\in  \R^k$  by
\begin{equation} \label{yat}
 y_j  = (y_{a,t})_j :=
\begin{cases} 
\log( t) + \tfrac{2}{n}\log|\mathrm{Norm}_{L/\Q}(a)|&  \mathrm{if}\ j=1,\vspace{.15cm}
\\   \frac2n\sum_{v\in\AL} e_v q_{jv}\log \lvert a \rvert_v   &\mathrm{if} \  2\le j\le k.
\end{cases}
\end{equation}
Then, with $\mathcal{L}:= \sqrt{\det(Q_{\phantom{l}}^\intercal Q)}/\big(2^{r_1}(2\sqrt{\pi})^{r_2}\pi^k\big)$, for any $\sigma\in\mathcal{D}$ we have 
\begin{align} \label{psipiece} 
&\int_{x\in\E}\e^{- t\,\|ax\|^2}\,d\mu_\E(x)  =   \frac{\mathcal{L}}{i^k}\int_{s\in I_\sigma}  \exp\!\big({\alpha}(s)-n\textstyle\sum_{j=1}^k y_j s_j\big)  \,ds,\\  \label{derivpiece} 
&   \frac{2t }{n}\int_{x\in\E} \|ax\|^2 \e^{- t\,\|ax\|^2}\,d\mu_\E(x) =  \frac{2\mathcal{L}}{i^k}\int_{s\in I_\sigma}  s_1\exp\!\big({\alpha}(s)-n\textstyle\sum_{j=1}^k y_j s_j\big)  \,ds.
\end{align}
\end{corollary}
\begin{proof} 
Define $r=r_{a,t}\in G:=\R_+^\AL$ by $r_v:=t^{1/2}|a|_v$.
In view of \eqref{Psipsi} and Lemma \ref{psiinversemellin}, \eqref{psipiece} will follow  from 
$\big(\delta(r)\big)_j=\e^{ny_j /2}$. Indeed, by \eqref{delta}, 
\begin{equation*}
\big(\delta(r)\big)_j
               := \prod_{v\in\AL} \big(t^{1/2} \lvert a \rvert_v\big)^{e_v q_{jv}} 
							= t^{\frac12\sum_v e_v q_{jv}}  \prod_{v\in\AL}  \lvert a \rvert_v^{e_v q_{jv}}.  
\end{equation*}
If $j=1$, then by   \eqref{qjs} we have $q_{jv}=1$ for all $v\in\AK$. Using \eqref{Norm} and  \eqref{sumev} we find
$$
\big(\delta(r)\big)_1=t^{n/2}|\mathrm{Norm}_{L/\Q}(a)|=\e^{ny_1/2}.
$$
  If  $j>1$, then $\sum_v e_v q_{jv}=0$  $\big($see \eqref{qjs}$\big)$, so 
$$
\big(\delta(r)\big)_j= \prod_{v\in\AL}  \lvert a \rvert_v^{e_v q_{jv}}=\e^{ny_j/2},
$$
as claimed.
To prove \eqref{derivpiece}, apply $-\frac{2t}{n}\frac{d}{dt}$ to   \eqref{psipiece}, noting that  $\frac{dy_j}{dt}=0$ for $j\ge2$.
\end{proof}

\section{Existence and uniqueness of the critical point}\label{CriticalPoint}
 We shall show that for every  $y\in\R^k$  there is a unique $\sigma=\sigma(y)\in\D\ \,\big($see \eqref{D}$\big)$ which is a critical point of $F_y\colon \D\to\R$, defined as
\begin{equation}\label{Obj}
F_y(\sigma):={\alpha}(\sigma)-\sum_{j=1}^k y_j \sigma_j = {\alpha}(\sigma) - y \cdot\sigma,
\end{equation}
with ${\alpha}$ as in \eqref{alpha}.
The map  taking $y\in\R^k$ to the critical point $\sigma(y)\in\D$ is   closely related to the Legendre transform of ${\alpha}\colon \D\to\R$, but we will develop the theory from scratch as ours is an easy case of the general theory of  the Legendre transform \cite[\S{E}]{HUL} \cite[\S1 and \S5]{Sim}.

\begin{lemma}\label{alphasteep} Let ${\alpha}\colon \D\to\R$ be as in \eqref{alpha}. Then ${\alpha}$ is steep \cite[p.\ 30]{Sim}, \ie
$$
\lim_{\|\sigma\|\to\infty}\frac{{\alpha}(\sigma)}{\|\sigma\|}=+\infty,
$$
where the limit is taken over $\sigma\in\D$ as its Euclidean norm $\|\sigma\|$ tends to infinity.
\end{lemma}
\begin{proof}
Recall that the linear map $S$ in  \eqref{S} is injective. Hence there exists $C>0$ such that, for all $\sigma\in\D$,
\begin{equation*}
\max_{v\in\AL} \big\{ S_v(\sigma) \big\} = \max_{v\in\AL} \big\{ \big| S_v(\sigma) \big| \big\} =: \|S(\sigma)\|_\infty \ge C\|\sigma\|.
\end{equation*}
For any $\sigma\in\D$, there is a $v_0 = v_0(\sigma)\in\AL$ such that $S_{v_0}(\sigma) = \max_{v\in\AL} \big\{ S_v(\sigma) \big\}$. The previous inequality says that
\begin{equation}
\label{Csigma}
S_{v_0}(\sigma) \ge C\|\sigma\|.
\end{equation}
 The known behavior of $\Gamma(z)$ for $z>0$ shows that there is a $\kappa<0$  such that
\begin{equation}\label{Lowerbound}
\log\Gamma_v(z)>\kappa\qquad\qquad\qquad\qquad\big( \Gamma_v  \text{ as  in }\eqref{Gammav}\big),
\end{equation}
 for all $z>0$ and all $v\in\AL$ ($\kappa=-1/5$ will do). Also, Stirling's formula   shows that
\begin{equation}\label{Needit}
\log\Gamma_v(z)> \frac{z\log z}{2}
\end{equation}
for $z\gg0$.
It follows from \eqref{Lowerbound}, \eqref{Csigma}, and \eqref{Needit} that when $\|\sigma\|$ is large,
\begin{equation*}
{\alpha}(\sigma) :=  \sum_{v\in\AL}\log\Gamma_v(S_v(\sigma))             
                        > n\kappa + \log\Gamma_{v_0}\big(S_{v_0}(\sigma)\big) 
												> n\kappa + \tfrac12 C\|\sigma\| \log(C\|\sigma\|),
\end{equation*}
and the lemma follows.
\end{proof}

The  next lemma amounts to the fact that the gradient $\nabla\! f$ of a steep and differentiable strictly convex function $f$ is a bijection.  However, in our case the domain $\D\not=\R^k$, which   means that we would   need to check the boundary behavior of ${\alpha}$  before  citing results from convex analysis. We prefer not to quote and instead adapt the usual proof \cite[\S1]{Sim} \cite[\S{E}]{HUL} to our  nicely behaved function ${\alpha}$.

\begin{lemma}\label{Criticalpoint} For any $y\in\R^k$ there is a unique $\sigma=\sigma(y)\in\D$ such that
$y= \nabla\!{\alpha}(\sigma) $.
\end{lemma}
\begin{proof}
For any $y\in\R^k$, let $F_y\colon \D\to\R$, $\,F_y(\tau):=\alpha(\tau)-y\cdot\tau$, and let 
\begin{equation}\label{dagger}
{\alpha}^\dagger(y):=\inf_{\tau\in\D}\big\{F_y(\tau)\big\} ,
\end{equation}
which we will now prove to be finite, \ie  ${\alpha}^\dagger(y)\not=-\infty$. Let $\tau^{(i)}$ be a sequence in $\D$ such that $F_y(\tau^{(i)})$ converges to ${\alpha}^\dagger(y)$.
By \eqref{Lowerbound}, ${\alpha}(\tau^{(i)})$ is bounded below, so it suffices to check that the sequence $\tau^{(i)}$ is bounded.   
By Lemma \ref{alphasteep}, ${\alpha}(\tau)>(\|y\|+1)\|\tau\|$ for $\tau\in\D$ with $\|\tau\|$ sufficiently large. For such $\tau$,
$$
F_y(\tau)>(\|y\|+1)\|\tau\|-\|y\|\,\|\tau\|=\|\tau\|,
$$
which shows that $\tau^{(i)}$ is bounded.  

We now prove that the infimum defining ${\alpha}^\dagger(y)$ is assumed at a point  in the open set   $\D\subset\R^k$.   Passing to a subsequence of the bounded sequence $\tau^{(i)}$, we may assume that
the $\tau^{(i)}\in\D$ converge to a point $\sigma$ in the closure of $\D$ in $\R^k$. Recall from    \eqref{D} that  $\D$ is the (non-empty) open set consisting of $\tau\in \R^k$ such that  $S_v(\tau)>0$ for all $v\in\AL$. If $\sigma\notin\D$, then $S_v(\sigma)=0$ for some $v\in\AL$. Since $\log\Gamma_v\big(S_v(\tau^{(i)})\big)\to+\infty$ as $S_v(\tau^{(i)}) \to 0^+$, and the remaining summands in the definition of ${\alpha}$ remain bounded from below (as does $y\cdot\tau^{(i)}$),
we conclude that $\sigma\in\D$.
Since  $\sigma$ is an interior minimum of the smooth function $F_y$, we have $\nabla\! F_y(\sigma)=0$. By \eqref{Obj},   $y= \nabla\!{\alpha}(\sigma) $, as claimed.

 To prove the uniqueness of $\sigma$, it suffices to prove that $F_y$ is a strictly convex function on $\D$.\footnote{\  That is, $F_y(t\tau+(1-t)\tilde\tau)<t F_y(\tau)+(1-t)F_y(\tilde\tau)$ for all $t\in(0,1)$ and all $\tau\not=\tilde\tau\in\D$. Such a function cannot have more than one critical point. To prove this, let $g(t):=F_y\big(t\tau+(1-t)\tilde\tau \big)$. Assuming  that $F_y$ is strictly convex, $g$ is   a strictly convex function of a single real variable   $t\in[0,1]$. Thus, $g^{\prime\prime}\ge0$, so $g$  has an increasing derivative $g^\prime(t)=\nabla\! F_y(t\tau+(1-t)\tilde\tau)\cdot(\tau-\tilde\tau) $. But $\nabla\! F_y(\tau)=0=\nabla\! F_y(\tilde\tau)$ would imply $g^\prime(0)=0=g^\prime(1)$, whence  $g$ is constant and therefore not strictly convex.}
The strict convexity of $F_y$ follows from the strict convexity of $\log\Gamma(z)$ for $z>0$. Indeed, 
\begin{align*}
F_y(t\tau+&(1-t)\tilde\tau)=-(t\tau+(1-t)\tilde\tau)\cdot y+{\alpha}(t\tau+(1-t)\tilde\tau)\\ &
=-(t\tau+(1-t)\tilde\tau)\cdot y+\sum_{v\in\AL}\log \Gamma_v\big(S_v(t\tau+(1-t)\tilde\tau)\big)\\ &
\le -(t\tau+(1-t)\tilde\tau)\cdot y+\sum_{v\in\AL}t\log\Gamma_v\big(S_v(\tau)\big)+(1-t)\log\Gamma_v\big(S_v(\tilde\tau)\big)\\ &
=tF_y(\tau)+(1-t)F_y(\tilde\tau),
\end{align*}
with strict inequality holding for $t\in(0,1)$ unless $S_v(\tau)=S_v(\tilde\tau)$ for all $v\in\AL$. But this  is impossible because  $S$ in \eqref{S} is injective.
\end{proof}
 The function  ${\alpha}^\dagger$ in \eqref{dagger} is a concave function of $y\in \R^k$, being  the infimum over $\tau\in\D$ of the set of concave (in fact, affine) functions  $y \mapsto -y\cdot\tau+{\alpha}(\tau)$. The convex function
 $-{\alpha}^\dagger$ is known as the Legendre transform of ${\alpha}$.

\section{Inequalities at the critical point}\label{IneqCrit}
To take advantage of the  inequality    \eqref{MainIneq}, we will later need to drop all terms  in   \eqref{MainIneq} corresponding to algebraic integers $a\not=1$. For this we will need some control of the first coordinate $\sigma_1(y)$ of the
function $\sigma$ in Lemma \ref{Criticalpoint}. In this subsection we take advantage of the concavity of $\Psi:=\Gamma^\prime/\Gamma$ to find a lower bound for $\sigma_1(y)$. Then we use the convexity of $\log\Gamma$ to find a lower bound for ${\alpha}\big(\sigma(y)\big)$.
Let 
\begin{equation}\label{digammaC}
\digamma_\C(z) := \digamma(z) + \digamma(z+\tfrac12),
\end{equation}
\begin{equation}\label{Digamma}
  \digamma_v(z) := \begin{cases}    \digamma(z) & \text{ if $v$ is real,}    \\
                                 \digamma_\C(z) & \text{ if $v$ is complex,} \end{cases} \qquad (v\in\AL).
\end{equation}
These definitions ensure that $\digamma_v(z) = \tfrac{d}{dz} \log\Gamma_v(z) = \Gamma_v'(z)/\Gamma_v(z)$ \big(see \eqref{Gammav}\big). Note that $\digamma_v(z)$ is a concave function of $z$ for $z>0$. We also note that $\digamma_v\colon (0,\infty)\to\R$ has an inverse function $\digamma_v^{-1}\colon \R\to(0,\infty)$  since $\digamma(z)$ is strictly increasing when $z>0$, tends to $-\infty$ as $z\to0^+$, and tends to $+\infty$ as $z\to+\infty$.

Writing out the $\ell$-th coordinate of the equation $y=\nabla\!{\alpha}(\sigma)$ in Lemma \ref{Criticalpoint}, we get
\begin{equation}\label{yj}
y_\ell = \sum_{v\in\AL}\digamma_v\big(S_v(\sigma)\big) q_{\ell v}
		   \qquad \big(S_v(\sigma)=\textstyle\sum_{j=1}^k q_{j v}\sigma_j, \quad \sigma:=\sigma(y) \big),
 \end{equation}
which for $\ell=1$ simplifies to
\begin{equation}\label{y1}
y_1=\sum_{v\in\AL}\digamma_v\big(S_v(\sigma)\big).
\end{equation}

\begin{lemma}\label{any field}
 Let $L$ be a  number field of degree $n$,  with $r_2$ complex places. For $y=(y_1,y_2, \ldots, y_k)\in\R^k$, let $\sigma_1(y)$ be the first coordinate of the function $ \sigma(y)$ defined in Lemma \ref{Criticalpoint}. Then
\begin{equation}\label{Goodforr2smallineq}
 \sigma_1(y_1,y_2, \ldots, y_k) \ge \digamma^{-1}\left(\frac{y_1}{n}\right)-\frac{r_2}{2n}.
\end{equation}
\end{lemma}
\begin{proof}
We prove \eqref{Goodforr2smallineq}  using  the concavity of $\digamma$. Namely, from \eqref{y1},
\begin{align*}
y_1 &= \sum_{v\in\AL}\digamma_v\big(S_v(\sigma)\big) 
= \sum_{v\in\AL}\digamma\big(S_v(\sigma)\big) + \displaystyle\sum_{v\ \text{complex}}\!\!\digamma\big(\tfrac12+S_v(\sigma)\big)\\ 
		&\le \textstyle n\digamma\bigg(\!\frac{1}{n}\Big(\sum_{v\in\AL} S_v(\sigma)  + \sum_{v\ \text{complex}}\big(\frac12+ S_v(\sigma)\big) \Big) \bigg) \\ 
		&=n\digamma\big(\!\tfrac{1}{n}\textstyle {\sum_{v\in\AL} e_v S_v(\sigma) + \frac{r_2}{2n}}\big)
		 =n\digamma\big(\sigma_1+ \frac{r_2}{2n}\big),
\end{align*}
where the last step uses 
\begin{equation}\label{nice}
\frac1n \sum_{v\in\AL} e_v S_v(\sigma) 
                                  = \sigma_1 \qquad \qquad \big(\sigma=(\sigma_1,\sigma_2, \ldots, \sigma_k)\in\C^k\big),
\end{equation}
which follows from \eqref{qjs} since 
\begin{equation*}
\sum_{v\in\AL} e_v S_v(\sigma) = \sum_{v\in\AL}\sum_{j=1}^k e_v q_{jv}\sigma_j
                          = \sum_{j=1}^k \sigma_j \sum_{v\in\AL} e_v q_{jv}
													= \sigma_1\sum_{v\in\AL} e_v
													= \sigma_1 n.
\end{equation*}
Inequality \eqref{Goodforr2smallineq} now follows, since $\digamma^{-1}$ is an increasing function.
\end{proof}

Our next result is a similar inequality for ${\alpha}(\sigma)$. 
 \begin{lemma}\label{alphaineq}
With notation as in Lemma \ref{any field}, we have
  \begin{equation}\label{alphaGoodforr2smallineq}
 {\alpha}(\sigma) \ge n \log \Gamma\big(\sigma_1+\textstyle\frac{r_2}{2n}\big)\qquad\qquad\qquad\qquad(\sigma=(\sigma_1, \ldots, \sigma_k)\in\D).
\end{equation}
\end{lemma}
\begin{proof} We compute directly from the definition \eqref{alpha} of ${\alpha}$, using the convexity  of $z\mapsto\log\Gamma(z)$ for $z>0$ and \eqref{nice}:
\begin{align*}
{\alpha}(\sigma) &= \sum_{v\in\AL}\log\Gamma\big(S_v(\sigma)\big) + \sum_{v\ \text{complex}}\log\Gamma\big(S_v(\sigma)+\tfrac12\big)\\ 
                       &\ge n\log\Gamma\bigg( \tfrac1n \Big(\textstyle\sum_{v\in\AL} S_v(\sigma) + \textstyle\sum_{v\ \text{complex}}\big(\tfrac12+S_v(\sigma)\big)\Big)\bigg)\\ 
											 &= n\log\Gamma\big( \textstyle\frac1n\sum_{v\in\AL} e_v S_v(\sigma)  +  \frac{r_2}{2n}\big) 
											  = n\log\Gamma\big(\sigma_1+ \frac{r_2}{2n}\big). \qedhere
\end{align*}
\end{proof}
 
We now prove a lower bound for $S_v(\sigma)$ in terms of $\sigma_1$ and $y_1$.

\begin{lemma}  \label{BoundSvl}
Let $u\in\AL$, $y\in\R^k$, and let  $\sigma:=\sigma(ny)\in \mathcal{D}$ be as in Lemma  \ref{Criticalpoint}.
Assume that $ y_1\ge t_0$ for some $t_0\in\R$, and  $n:=[L:\Q]\ge2$. Then     $S_u(\sigma) \ge2/5$ or
\begin{align} \label{Sv1}
S_u(\sigma) \ge\frac{1}{    (n-1)\Psi\big(  \frac{n\sigma_1}{n-1 }+\frac{1}{2} \big) - n t_0}\ge\frac{1}{    (n-1)\log(2\sigma_1+\tfrac12) - n t_0}>0.
\end{align}
\end{lemma}
\begin{proof}   
We shall show below that both denominators in \eqref{Sv1} are positive if $S_u(\sigma)< 2/5$, as we may assume.
Replacing $y$ with $ny$ in \eqref{y1}, we have
$$
ny_1=\sum_{v\in\AL}\Psi\big( S_v(\sigma)\big)+\sum_{\substack{v\in\AL\\ v\ \mathrm{complex}}}\\ \!\!\!\Psi\big(\textstyle{\frac12}+ S_v(\sigma)\big).
$$
Since $-\Psi$ is a monotone decreasing convex function on $(0, \infty)$, we find
\begin{align*}
 \Psi\big(S_u(\sigma)\big)&=ny_1 -
\sum_{\substack{v\in\AL\\ v\not=u}} \Psi\big(S_v(\sigma)\big)- \sum_{\substack{v\in\AL\\ v\ \mathrm{complex}}}  \!\!\!\Psi\big(\textstyle{\frac12}+
S_v(\sigma)\big)
\\ &
\ge ny_1 - (n-1)   \Psi\bigg( \frac{1}{n-1}\Big(\sum_{\substack{v\in\AL\\ v\not=u}}  S_v(\sigma) + \sum_{\substack{v\in\AL\\ v\ \mathrm{complex}}}
 \!\!\!\textstyle{\frac12}+  S_v(\sigma)  \Big) \bigg)\\ &
= ny_1 - (n-1)   \Psi\bigg( \frac{1}{n-1}\Big(-S_u(\sigma)+\sum_{v\in\AL}  e_v S_v(\sigma) + \sum_{\substack{v\in\AL\\ v\ \mathrm{complex}}}
 \!\!\! \textstyle{\frac12}  \Big) \bigg)
\\ &
=   ny_1 - (n-1)   \Psi\Big(  -\frac{S_u(\sigma)}{n-1}+\frac{n\sigma_1}{n-1 }+\frac{r_2}{2(n-1)} \Big) \qquad\big(\text{see}\ \eqref{nice} \big)
\\ &
\ge ny_1 - (n-1)    \Psi\Big(   \frac{ n\sigma_1}{n-1 }+\frac{r_2}{2(n-1)} \Big)
\ge ny_1 - (n-1)    \Psi\Big(   \frac{ n\sigma_1}{n-1 }+\frac12 \Big).
\end{align*}

From $x\Gamma(x)=\Gamma(x+1)$ and the fact that $\Psi(x)<0$ for $x<1.461$,
$$
\Psi(x)=-\frac1x+\Psi(1+x)<-\frac1x\qquad\qquad(x<0.461).
$$
Hence, as we are assuming $S_u(\sigma)<2/5$,
$$
\frac{-1}{S_u(\sigma)}> \Psi\big(S_u(\sigma)\big)\ge  ny_1 -  (n-1)  \Psi\Big(   \frac{ n\sigma_1}{n-1 }+\frac12 \Big)\ge  nt_0 -
 (n-1)  \Psi\Big(   \frac{ n\sigma_1}{n-1 }+\frac12 \Big).
$$
Since $ S_u(\sigma)>0$, the right-hand side above is negative. Hence the left-most inequality in   \eqref{Sv1} is proved.

Next  recall \cite[\S71, eq.\ (11)]{Ni},
$$ \log(x)-\Psi(x)=\frac1{2x}+2\int_0^\infty\frac{t}{(t^2+x^2)(\e^{2\pi t}-1)}
\,dt>0  \qquad \qquad(x>0).
$$
Whence $\Psi(x)<\log(x)$ for $x>0$, and so
$$
\Psi\Big(  \frac{n\sigma_1}{n-1 }+\frac12\Big)    < \log\!\Big(  \frac{n\sigma_1}{n-1 }+ \frac12 \Big) 
                                               \leq \log\!\big( 2\sigma_1 + \tfrac12\big).
$$
Now the second inequality in \eqref{Sv1} follows as before.
\end{proof}

\section{Asymptotics}\label{Asymptotics}
With a  view to applying Corollary \ref{sigma1ineq} and  the Basic Inequality \eqref{MainIneq},  in this section we will estimate integrals of the type 
\begin{equation}\label{fintegral}
  \frac{1}{i^k} \int_{s\in I_\sigma}  \e^{{\alpha}(s)-ny\cdot s} \,ds=  \int_{T\in \R^k}  \e^{{\alpha}(\sigma+iT)-ny\cdot (\sigma+iT)}\,dT=:\int_{ \R^k}\GT \,dT,
\end{equation}
where $n:=[L:\Q],\  y= (y_1, \ldots, y_k) \in\R^k,\ \sigma:=\sigma(ny)\in\D\subset\R^k$   as  in Lemma \ref{Criticalpoint}, and  $y\cdot s:=\sum_{j=1}^k y_j s_j$. We will let $ \M(T)$ be a Gaussian approximating $\GT$   (see  \eqref{HT}   
 below) in a   bounded neighborhood $\Delta\subset\R^k$  of   $T=0$ $\big($see \eqref{Deltadef}$\big)$. As usual with  the saddle point method, we decompose the integral \eqref{fintegral} into four pieces
 \begin{align}\nonumber
\int_{ \R^k}\GT \,dT=  &\int_{\R^k} \M(T) \,dT 
				 + \int_{\R^k-\Delta}  \GT \,dT - \int_{\R^k-\Delta}  \M(T) \,dT
				 \\ & \label{idea} \  +\int_\Delta \big(\GT - \M(T)\big)\,dT
=:   I_1+ I_2- I_3+ I_4.
\end{align}
The term  $I_1$ (\ie  $\int_{\R^k}\M$) is readily computed and gives  (as we will prove in this section) the main term in \eqref{idea}. Thus, we shall prove that the  terms  $I_2, I_3$ and $I_4$  are $o(I_1)$ as  $[L:K]\to\infty$, uniformly in $y\in\R^k$.

{\it{From now on we always (and usually tacitly) assume that the relative units $\ELK\subset E\subset \OLU$ for some subfield $K\subset L$}} $\big($see\eqref{RelUnits}$\big)$.  Define $\Log\colon L^*\to\R^\AL$ by
$$
\big(\Log(a)\big)_v:=\log(|a|_v)\qquad\qquad\qquad(a\in L^*,\ v\in\AL).
$$
Note that the complex places do not carry a factor of 2. Instead we
use this factor in the inner product \eqref{InnerProd} on $\R^\AL$ defined by $\langle\beta,\gamma\rangle :=\sum_{v\in\AL}e_v\beta_v\gamma_v$. 
The usefulness  of assuming $\ELK\subset E$   lies in the following.
\begin{lemma} \label{Placerestrict}
Suppose  $\ELK\subset E\subset\OLU$  and $q=(q_v)_{v\in\AL}\in \Log(E)^\perp$ lies in the orthogonal complement of $ \Log(E)$ inside $\R^\AL$ with respect to the above inner product. Then $q_v=q_{v^\prime}$ whenever $v$ and 
$v^\prime$ lie above the same place of $K$ and
\begin{equation}\label{kBound}
 1\le  k:=\dim_\R\!\big( \Log(E)^\perp\big)\le |\AK|\le[K:\Q] .
\end{equation}
\end{lemma} 
\begin{proof}
The lemma will follow from the fact that $\Log(E)^\perp$ is contained in the $\R$-span of $\Log(K^*)$ in $\R^\AL$. Clearly $\Log(E)^\perp \subset \Log(\ELK)^\perp$, so it suffices to prove that $\Span \Log(K^*) = \Log(\ELK)^\perp$. This follows from $\Span \Log(K^*) \subset \Log(\ELK)^\perp$ and $\dim(\Span \Log(K^*)) = \dim(\Log(\ELK)^\perp) = |\AK|$.
\end{proof}
\noindent Recall that in \eqref{qjs} we fixed a basis $q_1,...,q_k$ of $ \Log(E)^\perp$ such that $q_{1v}=1$ for all $v\in\AL$ and $\langle q_1,q_j\rangle=0$ for $2\le j\le k$.  In view of Lemma \ref{Placerestrict}, we will write $q_{jw}:=q_{jv}$ for any $v\in\AL$ extending $w\in\AK$.

 For a place $w\in\AK$, let $r_{1,w}$ and $r_{2,w}$ denote respectively the number of real and complex places of $L$ extending $w$, and let (cf.\ \cite[p.\ 134]{FS})
\begin{equation}\label{FSnotation}
m_w := r_{1,w} + 2r_{2,w},\ \kappa_w := \frac{r_{1,w} + r_{2,w}}{m_w},\ 
\alpha_\kappa(z) := \kappa\log \Gamma(z) + (1-\kappa) \log \Gamma(z+\half).
 \end{equation}
Note that  $m_w =e_w[L:K]=   [L:K]$ or $2[L:K]$, and that $\frac12\le\kappa_w\le1$.

  Lemma \ref{Placerestrict} implies that $S_v$  defined in \eqref{S} satisfies 
\begin{equation}\label{Sw}
S_v(s): =\sum_{j=1}^k  q_{jv} s_j=\sum_{j=1}^k  q_{jw} s_j=:S_w(s)\qquad\qquad(s\in\C^k),
 \end{equation}
where    $v\in\AL$ is any place extending $ w\in\AK$. We  therefore rewrite   ${\alpha}$ in \eqref{alpha} as
\begin{equation}\label{Walpha}
{\alpha}(s) := \sum_{v\in\AL} \log \Gamma_v\big(S_v(s)\big) 
                     = \sum_{w\in\AK} \sum_{v\mid w} \log \Gamma_v\big(S_w(s)\big) 
										 = \sum_{w\in\AK} m_w \alpha_{\kappa_w}\!\big(S_w(s)\big),
 \end{equation}
where we write $v\mid w$ if $v$ extends $w$, and $\alpha_{\kappa_w}$ was defined in \eqref{FSnotation}.

For each $w\in\AK$ and $\sigma\in\D\ \,\big(\text{see }\eqref{D}\big)$, define $\rho_w\colon \R^k \to \C$ by
\begin{equation}\label{rhow}
\rho_w(T) := \alpha_{\kappa_w}\!\big(S_w(\sigma+iT)\big) - \alpha_{\kappa_w}\!\big(S_w(\sigma)\big) - i\alpha_{\kappa_w}'\!\big(S_w(\sigma)\big)S_w(T) + \frac{1}{2} \alpha_{\kappa_w}''\!\big(S_w(\sigma)\big)\big( S_w(T)\big)^2,
 \end{equation}
\ie $\rho_w$ is the error in the degree-2 Taylor approximation of $T \mapsto \alpha_{\kappa_w}\!\big(S_w(\sigma+iT)\big)$ at $T=0$.  We shall henceforth take any  $y\in\R^k$ and let $\sigma:=\sigma(ny)$ be the corresponding saddle point  in Lemma \ref{Criticalpoint}. Thus $\nabla \alpha(\sigma)=ny$. Using this and \eqref{Walpha}, we find 
\begin{equation}\label{FirstOrder}
 \sum_{j=1}^k n y_j   T_j =  \sum_{j=1}^k  T_j\sum_{w\in\AK} m_w \alpha_{\kappa_w}^\prime\big(S_w(\sigma)\big)q_{jw}=\sum_{w\in\AK} m_w \alpha_{\kappa_w}^\prime\big(S_w(\sigma)\big)S_w(T).
\end{equation}
It follows from  \eqref{Walpha}--\eqref{FirstOrder}  that
\begin{align}
{\alpha}(\sigma + iT) - ny\cdot(\sigma+iT) &= {\alpha}(\sigma) - ny\cdot\sigma -  \frac{1}{2} \sum_{w\in\AK} m_w \alpha_{\kappa_w}''\!\big(S_w(\sigma)\big)S_w(T)^2 + \rho(T),\nonumber\\
 \label{rhoT}
\rho(T) &:= \sum_{w\in\AK} m_w \rho_w(T).
 \end{align}
The linear terms in $T$  have disappeared as  $\sigma$ is a critical point of $s \mapsto \alpha(s)-ny\cdot s$.
 
For fixed $y\in\R^k$ and $\sigma:=\sigma(ny)\in\D$, define the following functions of $T\in\R^k$:
\begin{align}\label{HT}
\M(T) &:= \e^{{\alpha}(\sigma)-ny\cdot\sigma -\frac{1}{2}H(T)},\\ \label{HsigmaT}
 H(T) &:= \sum_{w\in\AK} m_w \alpha_{\kappa_w}''\!\big(S_w(\sigma)\big) S_w(T)^2,\\
\GT   &:= \e^{{\alpha}(\sigma+iT)-ny\cdot(\sigma+iT)}=e^{\rho(T)}\M(T).\label{GT}
\end{align}
Although $\M, H, \mathcal{G}$ and $\rho$ depend on $y\in\R^k$,   we do not include $y$ in our  notation.
\subsection{The main term}\label{MAINTERM}
In Lemma \ref{ccomputation} we defined the $|\AL|\times k$ matrix $Q$ of rank $k$ whose coefficients are  $Q_{v,j} := q_{jv}$. 
We will write $\mathcal{Q}$ for the $|\AK|\times k$ matrix with entries $\mathcal{Q}_{wj}:=q_{jw}$ and rank $k$.  Recall that
 we   write $q_{jw}:=q_{jv}$ for any $v\in\AL$ extending $w\in\AK$. 
Let $\AKk$ be the set of $k$-element subsets of $\AK$.  For $\eta\in \AKk$, let $\mathcal{Q}_\eta$ be the $k\times k$ submatrix of $\mathcal{Q}$   whose rows are indexed by the elements of $\eta$. 
In the computation of $\psi(\chi)$ in Lemma \ref{psiinversemellin}  the term $\det(Q_{\phantom{l}}^\intercal Q) $ appears. Using the smaller matrix $\mathcal{Q}$ we have 
\begin{equation}\label{QcalQ}
 \det(Q_{\phantom{l}}^\intercal Q)=\det(\mathcal{Q}_{\phantom{l}}^\intercal \mathcal{Q})\prod_{w\in\AK}(r_{1,w}+r_{2,w})\qquad\qquad\big(r_{1,w},r_{2,w} \text{ as in }\eqref{FSnotation} \big),
\end{equation}
as follows from
\begin{align*}
(Q_{\phantom{l}}^\intercal Q)_{i,j} =\sum_{v\in\AL}q_{iv}q_{jv}= \sum_{w\in\AK}q_{iw}q_{jw}\sum_{ v|w} 1=
\sum_{w\in\AK} q_{iw}q_{jw}(r_{1,w}+r_{2,w}) .
\end{align*}

 Next we calculate  some integrals such as $I_1$ in \eqref{idea}, and  its derivatives.

\begin{lemma}
\label{Lem:bw}Let 
 $\mathcal{Q}$ and  $\mathcal{Q}_\eta$ be as  above,  where $\eta\in \AKk$, let $(b_w)_{w\in\AK} \in \R_+^{\AK}$, and  define
\begin{equation}\label{CalD}
\dD_{ {\eta}} := {\det}^2(\mathcal{Q}_{ {\eta}}) \prod_{w\in {\eta}} b_w ,\qquad\qquad
\dD := \sum_{ {\eta}\in\AKk} \dD_{ {\eta}}.
\end{equation}
Then, with   $S_w$ as in \eqref{Sw}, 
\begin{equation}\label{Gaussian}
\int_{T\in\R^k} \exp\!\Big(\!-\frac{1}{2} \sum_{w\in\AK} b_w S_w(T)^2 \Big) \,dT = (2\pi)^{k/2}\dD^{-1/2}.
\end{equation}
Furthermore,  for any $w_0\in\AK$ we have
\begin{align*}
\int_{\R^k} S_{w_0}(T)^4 \exp\!\Big(\!-\frac{1}{2} \sum_{w\in\AK} b_w S_w(T)^2\Big) \,dT
      &=    3(2\pi)^{k/2}\dD^{-5/2} b_{w_0}^{-2} \Big(\sum_{ {\eta} \ni w_0} \dD_{ {\eta}}\Big)^2 \\
			&\leq 3(2\pi)^{k/2}\dD^{-1/2}b_{w_0}^{-2}
\end{align*}
and
\begin{align*}
\int_{\R^k} S_{w_0}(T)^6 \exp\!\Big(\!-\frac{1}{2} \sum_{w\in\AK} b_w S_w(T)^2 \Big) \,dT 
      &=    15(2\pi)^{k/2}\dD^{-7/2} b_{w_0}^{-3}\Big(\sum_{ {\eta} \ni w_0} \dD_{ {\eta}}\Big)^3 \\
			&\leq 15(2\pi)^{k/2}\dD^{-1/2}b_{w_0}^{-3}.
\end{align*}
\end{lemma}
\begin{proof}
Let $P=(P_{w,j})$ be the $|\AK|\times k$ matrix with entries 
$ P_{w,j}:=\sqrt{b_w}q_{jw}\ \,(w\in\AK,\ 1\le j\le k).$ Then for $T=(T_1,...,T_k)\in\R^k$, considered as a $ k\times1 $ matrix,
$PT\in\R^{\AK}$ satisfies $(PT)_w=\sqrt{b_w}S_w(T)$. Hence
$$
\sum_{w\in\AK} b_w S_w(T)^2 =(PT)_{\phantom{1}}^\intercal PT=T_{\phantom{1}}^\intercal(P_{\phantom{1}}^\intercal P)T=T_{\phantom{1}}^\intercal H T\qquad\qquad(H:=P_{\phantom{1}}^\intercal P). 
$$
The $k\times k$ matrix $H$ is clearly positive semi-definite. The Cauchy-Binet formula  gives $\det(H)=\dD$, with $\dD$ as in \eqref{CalD}.\footnote{\ The Cauchy-Binet formula  computes  $\det(AB)$, where $A$ is  a $k\times \ell$     and $B$ is  $\ell\times k$, in terms of the  $k\times k$ minors of $A$ and $B$.} But $\dD>0$  as $\dD_\eta>0$ for at least one $\eta\in\AKk$, since  $\mathcal{Q}$ has rank $k$. Hence $H$ is positive definite, and so the integral in \eqref{Gaussian} is the well-known Gaussian integral attached to a positive definite quadratic form $H$ in $k$ variables,  as claimed in \eqref{Gaussian}.  

The other equalities in Lemma \ref{Lem:bw} are obtained by differentiating \eqref{Gaussian} with respect to $b_{w_0}$ repeatedly. Indeed,  noting that the partial derivative 
$\frac{\partial \dD}{\partial b_{w_0}} = b_{w_0}^{-1} \sum_{ {\eta} \ni w_0} \dD_{ {\eta}}$
is independent of  $b_{w_0}$, \ie 
$\frac{\partial^2\dD}{\partial b_{w_0}^2} =0$, 
we have 
\begin{align*}
-\frac{1}{2} \int_{\R^k} S_{w_0}(T)^2 \exp\!\Big(\!-\frac{1}{2} \sum_{w\in\AK} b_w S_w(T)^2 \Big) \,dT &= -\frac{1}{2} (2\pi)^{k/2}\dD^{-3/2} \Big(b_{w_0}^{-1} \sum_{ {\eta} \ni w_0} \dD_{ {\eta}}\Big), \\
 \frac{1}{4} \int_{\R^k} S_{w_0}(T)^4 \exp\!\Big(\!-\frac{1}{2} \sum_{w\in\AK} b_w S_w(T)^2 \Big) \,dT &=  \frac{3}{4} (2\pi)^{k/2}\dD^{-5/2} \Big(b_{w_0}^{-1} \sum_{ {\eta} \ni w_0} \dD_{ {\eta}}\Big)^2, \\
-\frac{1}{8} \int_{\R^k} S_{w_0}(T)^6 \exp\!\Big(\!-\frac{1}{2} \sum_{w\in\AK} b_w S_w(T)^2 \Big) \,dT &= -\frac{15}{8} (2\pi)^{k/2}\dD^{-7/2} \Big(b_{w_0}^{-1} \sum_{ {\eta} \ni w_0} \dD_{ {\eta}}\Big)^3,
\end{align*}
 proving  the equalities. The inequalities follow from $\sum_{ {\eta} \ni w_0} \dD_{ {\eta}} \leq \dD$, as    $\dD_{ {\eta}}\ge0$.
\end{proof}
As $\alpha_\kappa''(t)>0$ for $t>0$, we can now evaluate $I_1$.
\begin{corollary}\label{I1term}With notation as in   \eqref{HT}, for $y\in\R^k$ we have  
\begin{equation*}
I_1=I_1(ny):=\int_{\R^k} \M(T) \,dT = \frac{(2\pi)^{k/2} \e^{\alpha(\sigma)-ny\cdot\sigma}}
                               {\sqrt{\det\!\big(H(\sigma)\big)}},
\end{equation*}
where  $ \sigma:=\sigma(ny)\in \D$ as in Lemma \ref{Criticalpoint} and 
\begin{equation}\label{Eqn:detH}
\det(H(\sigma)) = \sum_{\eta\in \AKk} {\det}^2(\mathcal{Q}_\eta) \prod_{w\in \eta } m_w \alpha_{\kappa_w}''\!\big(S_w(\sigma)\big).
\end{equation}
\end{corollary}
\subsection{The small terms} We begin by quoting some one-variable estimates. 

\begin{lemma}
\label{Lem:A} If   $m\geq 1000$,  $\kappa\in[\half, 1]$, and  $r>0$, then 
\begin{align}\label{estint1}
\int_{-\infty}^\infty \lvert \e^{m\alpha_\kappa(r+it)}\rvert \,dt &< 1.0021 \frac{\sqrt{2\pi}\e^{m\alpha_\kappa(r)}}{\sqrt{m\alpha_\kappa''(r)}},\\ \label{estint2}
\int_{-\infty}^\infty| t \e^{m\alpha_{\kappa}(r+it)} | \,dt& < 0.83\frac{ \sqrt{2\pi} \e^{m\alpha_{\kappa}(r)}}
                                                                                  {m\alpha_\kappa''(r)}.
\end{align}
\end{lemma}
\begin{proof}
The estimate \eqref{estint1} is proved in \cite[Lemma 4.4]{Su2}. We now prove \eqref{estint2}.  From  \cite[Lemma 4.11]{Su2} we have
\begin{equation}\label{quotedlemma411}
\int_{-\frac{r}{3\sqrt{2}}}^{ \frac{r}{3\sqrt{2}}} | t \e^{m\alpha_\kappa (r+it)} | \,dt< \frac{72  \e^{m\alpha_{\kappa}(r)}}{35m\alpha_\kappa''(r)},
 \end{equation}
while from \cite[Lemma  5.3]{FS} we have
\begin{equation}\label{quotedlemma53}
\int_{|t|>\frac{r}{3\sqrt{2}}}  | t \e^{m\alpha_\kappa (r+it)} | \,dt< \frac{2r^2 \e^{m\alpha_{\kappa}(r)}}{m(\kappa-\frac2m)(1+\frac{1}{72})^{m\kappa\lfloor r\rfloor/2} (1+\frac{1}{18})^{(m\kappa-2)/2}},
 \end{equation}
where $\lfloor r\rfloor$ is the floor  of $r$. Since $0<r^2\alpha_\kappa''(r)<1+r$ \cite[p.\ 141]{FS}, 
we have 
$$
\frac{r^2  }{ (1+\frac{1}{72})^{m\kappa\lfloor r\rfloor/2} }\le\frac{1}{\alpha_\kappa''(r)}\frac{1+r}{ (1+\frac{1}{72})^{m\kappa\lfloor r\rfloor/2}}\le \frac{2}{\alpha_\kappa''(r)}.
$$
Indeed, for $0<r<1$ the last inequality is obvious, while for $r\ge1$ a much better inequality follows from 
   $m\kappa\ge500$.
Hence
$$
\int_{|t|>\frac{r}{3\sqrt{2}}}  | t \e^{m\alpha_\kappa (r+it)} | \,dt< \frac{1}{m\alpha_\kappa''(r)}\frac{4\e^{m\alpha_\kappa(r)}}{(\frac12-\frac{2}{1000}) (1+\frac{1}{18})^{(500-2)/2}}< \frac{0.00002\e^{m\alpha_\kappa(r)}}{m\alpha_\kappa''(r)}.
$$
Combining this with \eqref{quotedlemma411} we obtain \eqref{estint2}. 
\end{proof}

We will need the following inequality, proved by elementary calculus.
\begin{equation}\label{maxprod}
x^{5/2}\e^{-x}\le\Big(\frac{5}{2\e}\Big)^{5/2} < 0.8112 \qquad\qquad\qquad\big(x\geq0\big).
\end{equation}
\begin{lemma}
\label{Lem:D}
Suppose $m\ge1000,\ \frac12\le \kappa\le1,\ 0<D\le m^{1/3} \kappa$, and  let 
\begin{equation}\label{deltadef}
\delta := \frac{D}{m^{1/3}\sqrt{\alpha_\kappa''(r)}}.
\end{equation} Then, for any $r>0$,
\begin{equation}\label{int1est}
\int_{\lvert t \rvert > \delta} \lvert \e^{m\alpha_\kappa(r+it)}\rvert \,dt <\bigg( \frac{10^{-76} +\frac{41.43}{D^6}}{m}\bigg) \frac{\sqrt{2\pi}\e^{m\alpha_\kappa(r)}}{\sqrt{m\alpha_\kappa''(r)}},
\end{equation} 
and
\begin{equation}\label{int2est}
\int_{\lvert t \rvert > \delta} \e^{-\half m\alpha_\kappa''(r) t^2} \,dt <   \frac{3.67}{mD^6} \, \frac{\sqrt{2\pi}}{\sqrt{m\alpha_\kappa''(r)}}.
\end{equation} 
\end{lemma}
\begin{proof}
Inequality   \eqref{int2est} follows from
\begin{align*}
 \int_{\lvert t \rvert > \delta} \e^{-\half m\alpha_\kappa''(r) t^2} \,dt &\le \frac{2\e^{-m^{1/3} D^2/2}}{m^{2/3}D\sqrt{\alpha_\kappa''(r)}}=\frac{\sqrt{2\pi}}{\sqrt{m\alpha_\kappa''(r)}}\, \frac{8(m^{1/3} D^2/2)^{5/2} \e^{-m^{1/3} D^2/2}}{m\sqrt{\pi}D^6}\\
&< \frac{\sqrt{2\pi}}{\sqrt{m\alpha_\kappa''(r)}}\, \frac{3.67}{mD^6},
\end{align*}
where the first inequality is from  \cite[p.\ 139]{FS} and the last one  uses 
\eqref{maxprod} with $x:= m^{1/3} D^2/2$. 
To prove \eqref{int1est} we use  \cite[Lemma 4.5]{Su2},
\[
\frac{\int_{\lvert t \rvert > \delta} \lvert \e^{m\alpha_\kappa(r+it)}\rvert \,dt}{ \frac{1}{m} \frac{\sqrt{2\pi}\e^{m\alpha_\kappa(r)}}{\sqrt{m\alpha_\kappa''(r)}}} < 10^{-76} + \frac{2^{3/2} m^{5/6} \exp(-m^{1/3} D^2/4)}{\sqrt{\pi}D}<10^{-76} +\frac{41.43}{D^6},
\]
where the second inequality again follows from \eqref{maxprod}. 
\end{proof}

Next we deal  with the second order remainder term in the Taylor expansion about $a$ of $\log\Gamma(a+ib)$,   taking $a=S_w(\sigma)$ and $b=S_w(T)$.  
\begin{lemma}
\label{Lem:Rho}
For  $w\in\AK,\ \sigma\in\D$ $\big($see \eqref{D}$\big)$, $ T\in\R^k$ and $\rho_w $ as in \eqref{rhow}, we have
\begin{align}\label{claima}
\big| \im\big(\rho_w(T)\big) \big| &\leq -\frac{\alpha_{\kappa_w}^{(3)}\!\big(S_w(\sigma)\big)}{3!}\lvert S_w(T) \rvert^3
                              \leq \frac{\sqrt{2}}{3} \alpha_{\kappa_w}''\!\big(S_w(\sigma)\big)^{3/2} \lvert S_w(T) \rvert^3,\\
 \label{claimb}  \big|\re\big(\rho_w(T)\big) \big|&\leq  \frac{\alpha_{\kappa_w}^{(4)}\!\big(S_w(\sigma)\big)}{4!} S_w(T)^4
                              \leq \frac{1}{2}\alpha_{\kappa_w}''\!\big(S_w(\sigma)\big)^2 S_w(T)^4,\\ \label{claimc}
\im \big(\rho_w(-T)\big)&=-\im \big(\rho_w(T)\big),\qquad \re\big(\rho_w(-T)\big)=\re \big(\rho_w(T)\big),
\\ \label{claimd}
 \mathrm{if}\ |S_w(T) | &\leq S_w(\sigma),\  \mathrm{then}\ 0 \leq \re\big(\rho_w(T)\big)   \leq \frac{\alpha_{\kappa_w}''\!\big(S_w(\sigma)\big)}{4} S_w(T)^2  . 
\end{align}
\end{lemma}
\begin{proof}
The first inequalities in \eqref{claima} and \eqref{claimb} are proved in  \cite[Lemma 4.7]{Su2}, as is also \eqref{claimd}. The second   inequalities in  \eqref{claima} and \eqref{claimb} follow from \cite[Lemma 5.2]{FS}  and $\kappa_w \geq \half$. The identities in  \eqref{claimc} follow from \eqref{rhow} and  $\log\Gamma(\overline{z})=\overline{\log\Gamma(z)}$.
\end{proof}

\begin{lemma}$\big($\cite[(5.11)]{FS}$\big)$
\label{Lem:uv}
If $u,v\in\R$ with $0\leq u \leq R$, then
\[\lvert \re(e^{u+iv}-1)\rvert \leq \frac{v^2}{2} + u\frac{e^R-1}{R}.\]
\end{lemma}

We first estimate the easier ``outer"   terms, $I_2$ and $I_3$ in \eqref{idea}, \ie  where the region of integration is $\R^k-\Delta$.  For $y\in\R^k$,   
let $\eta_0=\eta_0(y)\in\AKk$ correspond to a maximal summand  in (\ref{Eqn:detH}), so 
\begin{equation}\label{eta0def}{\det}^2(\mathcal{Q}_{\eta}) \prod_{w\in\eta} m_w \alpha_{\kappa_w}''\!\big(S_w(\sigma)\big) \leq\,
  {\det}^2(\mathcal{Q}_{\eta_0}) \prod_{w\in\eta_0} m_w \alpha_{\kappa_w}''\!\big(S_w(\sigma)\big)\qquad(\forall\eta\in\AKk).
\end{equation}
Thus, 
\begin{equation*} 
\det\!\big(H(\sigma)\big) \leq \,|\AKk|\,{\det}^2(\mathcal{Q}_{\eta_0}) \prod_{w\in\eta_0} m_w \alpha_{\kappa_w}''\!\big(S_w(\sigma)\big),
\end{equation*}
and so
\begin{equation}\label{Comparison}
\frac{1}{\lvert \det(\mathcal{Q}_{\eta_0})\rvert \prod_{w\in\eta_0} \sqrt{m_w \alpha_{\kappa_w}''\!\big(S_w(\sigma)\big)}} \leq
\frac{\sqrt{|\AKk|}}{\sqrt{\det\!\big(H(\sigma)\big)}}.
\end{equation}
For  $y\in\R^k,\ w\in \eta_0(y)$ and $D>0$, let $\big(${\it{cf}}.\ \eqref{deltadef}$\big)$
\begin{equation}\label{deltaw}
\delta_w :=\frac{D}{m_w^{1/3}\sqrt{\alpha_{\kappa_w}''(S_w(\sigma))}}.
\end{equation} Define the neighborhood  $\Delta\subset\R^k$ of $T=0\in\R^k$ as 
\begin{equation}\label{Deltadef}
\Delta =\Delta(y):= \big\{T \in \R^k \big|\,\lvert S_w(T) \rvert < \delta_w \ (\forall w\in\eta_0)\big\}.
\end{equation}
 
  The next lemma shows that $I_2$ and $I_3$ are small compared to $I_1$ in Corollary \ref{I1term}.

\begin{lemma}\label{I2I3terms} Suppose $m:=[L:K]\ge1000,\ \,0<D<m^{1/3}/\sqrt{2}$, and $y\in\R^k$. Then, with $\Delta$ as in \eqref{Deltadef}, $ \sigma:=\sigma(ny)\in \D$  as in Lemma \ref{Criticalpoint},  $\M$ and $\mathcal{G}$  as in \eqref{HT} and  \eqref{GT},  we have 
\begin{align}\label{I2Est}
 |I_2|  &= \Big|\int_{\R^k-\Delta}   \GT \,dT \Big|\le\frac{1.0021^{k-1}\Big(10^{-76} +\frac{41.43}{D^6} \Big)k\sqrt{|\AKk|}} {m}I_1,
\\ \label{I3Est}
|I_3|  &= \Big|\int_{\R^k-\Delta}   \M(T) \,dT \Big| \le  \frac{3.67k\sqrt{|\AKk|}}{m D^6}\,I_1.
\end{align}
\end{lemma}
\begin{proof}
We first prove \eqref{I2Est}. Note that $\Gamma(z)=\int_0^\infty x^z\e^{-x}\frac{dx}{x}$ implies 
\begin{equation}\label{gammaineq}
|\Gamma(z)|\le\Gamma\big(\re(z)\big) \qquad\qquad\qquad(\re(z)>0).
\end{equation}
Using this,   \eqref{GT} and
 \eqref{Walpha} we have,  
\begin{equation}\nonumber
\int_{\R^k-\Delta} | \GT |\,dT \le 
\e^{-ny\cdot\sigma}\prod_{\substack{w\in\AKk\\ w\not\in{\eta}_0}} \e^{m_w \alpha_{\kappa_w}(S_w(\sigma))} \int_{ \R^k-\Delta} \Big| \prod_{w\in{\eta}_0} \e^{m_w \alpha_{\kappa_w}(S_w(\sigma+iT))}  \Big| \,dT.
\end{equation}
Let $\rectangle\subset\R^{\eta_0}$ denote the $k$-dimensional  box
\begin{equation}\label{rectangledef}
\rectangle =\rectangle(y):= \big\{\tilde{T} \in \R^{\eta_0} \big| \,\lvert\tilde{T}_w\rvert\le  \delta_w\quad (\forall w\in\eta_0)\big\},
\end{equation}
and let $B^c:=\R^{\eta_0}-B$ denote its complement. 
Making the change of variables $\tilde{T}_w := S_w(T)$  for $w\in{\eta}_0$, we have 
\[
 \int_{ \R^k-\Delta} \Big| \prod_{w\in{\eta}_0} \e^{m_w \alpha_{\kappa_w}(S_w(\sigma+iT))}  \Big| \,dT
=\frac{1}{ |\!\det(\mathcal{Q}_{\eta_0})|}
\int_{\tilde{T}\in \rectangle^c} \Big| \prod_{w\in{\eta}_0} \e^{m_w \alpha_{\kappa_w}(S_w(\sigma)+i\tilde{T}_w)}  \Big| \,d\tilde{T}.
\]
The latter  integral is easy to bound using Lemmas \ref{Lem:A} 
 and \ref{Lem:D}. We integrate over $k$ (overlapping) regions, each of which has $k-1$ of the $\tilde{T}_w$ range over all of $\R$, and the remaining $\tilde{T}_{w_0} $ over  $| \tilde{T}_{w_0} |> \delta_{w_0}$. Since $m_w \geq m:=[L:K]$, we conclude that 
\[
\int_{\R^k-\Delta} | \GT |\,dT \le 
  \frac{
k (2\pi)^{k/2} 1.0021^{k-1}\big(10^{-76} +\frac{41.43}{D^6}\big)\e^{\alpha(\sigma)-ny\cdot\sigma} 
}{
m |\!\det(\mathcal{Q}_{\eta_0})| \prod_{w\in\eta_0} \sqrt{m_w\alpha_{\kappa_w}''\!\big(S_w(\sigma)\big)}
} .
\]
 Now inequality  \eqref{Comparison}  and Corollary \ref{I1term} prove  \eqref{I2Est}.

Next we prove  \eqref{I3Est}. Changing variables as before, we have
\begin{align*}
|I_3|&=\e^{\alpha(\sigma)-ny\cdot\sigma}\int_{\R^k-\Delta}  \exp\!\Big(\!-\frac{1}{2} \sum_{w\in\AK} m_w \alpha_{\kappa_w}''\!\big(S_w(\sigma)\big) S_w(T)^2\Big)\,dT
\\ & \le
 \e^{\alpha(\sigma)-ny\cdot\sigma}\int_{\R^k-\Delta}  \exp\!\Big(\!-\frac{1}{2} \sum_{w\in\eta_0} m_w \alpha_{\kappa_w}''\!\big(S_w(\sigma)\big) S_w(T)^2\Big)\,dT\\ &
=
\frac{\e^{\alpha(\sigma)-ny\cdot\sigma} }{ |\!\det(\mathcal{Q}_{\eta_0})|}\int_{B^c}  \exp\!\Big(\!-\frac{1}{2} \sum_{w\in\eta_0} m_w \alpha_{\kappa_w}''\!\big(S_w(\sigma)\big)\tilde{T}_w^2\Big)\,d\tilde{T}
\end{align*}
Once again, we bound $\int_{B^c}$ using $k$ overlapping regions, one for each $w_0 \in \eta_0$. The integral over the region  given by all $\tilde {T}\in\R^{\eta_0}$ such that $|\tilde{T}_{w_0}|>\delta_{w_0}$ is bounded by 
\[
            \int_{\lvert \tilde{T}_{w_0} \rvert > \delta_{w_0}} 
				          \e^{-\frac12 m_{w_0} \alpha_{\kappa_{w_0}}''(S_{w_0}(\sigma))\tilde{T}_{w_0}^2} \,d\tilde{T}_{w_0} 
		\prod_{\substack{w \in \eta_0 \\ w \neq w_0}} 
		        \int_{-\infty}^\infty
						      \e^{-\frac12 m_{w} \alpha_{\kappa_w}''(S_{w}(\sigma))\tilde{T}_w^2} \,d\tilde{T}_{w}.\]
We can use  \eqref{int2est} to bound the first integral, and    the remaining integrals are explicitly known. Hence, summing over the $k$ regions,   
\[
|I_3|\le\frac{ (2\pi)^{k/2}\,\e^{\alpha(\sigma)-ny\cdot\sigma} }{|\!\det(\mathcal{Q}_{\eta_0}) |}
       \frac{3.67\,k}{m D^6} 
	     \prod_{w \in \eta_0 } \frac{1}{\sqrt{m_w\alpha_{\kappa_w}''\big(S_w(\sigma)\big)}}.
\]
We again conclude using  \eqref{Comparison}. 
\end{proof}

For the ``inner" integral   $I_4=\int_\Delta(\mathcal{G}-\mathcal{H})$ in \eqref{idea}, we can only expect estimates of the kind $O(I_1/m)$, whereas  $I_2$ and $I_3$ are essentially  $O\big(I_1\exp(-m^{1/3})\big)$. This allowed us to use simple estimates for the contribution of places $w\notin\eta_0$. However, to estimate $I_4$ we shall need the following geometric result. 

\begin{lemma}\label{JamesLemma}
Let $M = (m_{ij})$ be an $N \times k$ matrix of rank $k$, and let $a_i>0\ \,(1\le i\le N)$. Define linear maps $P_i\colon \R^k \to \R$ by $ P_i(T) := \sum_{j=1}^k m_{ij} T_j,$ where $T=(T_1,...,T_k)$.
For any $k$-element subset $\eta=\{i_1,\ldots,i_k\}\subset\{1,2,\ldots,N\}$, let $M_\eta$ denote the $k\times k$ submatrix of $M$ given by  $ \big(M_\eta\big)_{\ell,j}=m_{i_\ell,j}$. Define
$E_\eta := \lvert \det(M_\eta)\rvert \prod_{i\in\eta} a_i,$ 
 and let $\eta_0$ maximize $E_\eta$. Then 
\[a_i \lvert P_i(T) \rvert \leq \sum_{j\in\eta_0} a_j \lvert P_j(T) \rvert\qquad\qquad(1\leq i \leq N, \ T\in\R^k).\]
\end{lemma}
\begin{proof}
 Replacing  $m_{ij}$ with $a_i m_{ij}$, we may assume  $a_i = 1$. Hence $\eta_0$ simply maximizes $\lvert\det(M_\eta)\rvert$. Fix  $i\in\{1,2,\dotsc,N\}$, and define $\lambda_j\in\R$ for $j\in\eta_0$ by $ P_i = \sum_{j\in\eta_0} \lambda_j P_j.$ 
For  $j\in\eta_0$, let $M_{j}$ denote $M_\eta$ with the $j^{\text{th}}$ row of $M$ replaced by the $i^{\text{th}}$ row. Then, by Cramer's rule, 
$
\lvert \lambda_j \det(M_\eta) \rvert =  \lvert \det(M_j) \rvert \leq\lvert \det(M_\eta)\rvert,  
 $
so $\lvert \lambda_j \rvert \leq 1$. Hence 
\[\lvert P_i(T) \rvert = \Big| \sum_{j\in\eta_0} \lambda_j P_j(T) \Big| \leq \sum_{j\in\eta_0}| P_j(T) |. \qedhere\]
\end{proof}

\begin{lemma}\label{I4term} For $y\in\R^k$ and $D>0$  we have 
\begin{align}\label{I4Est}
 |I_4|  &= \Big|\int_\Delta  \big( \GT -\M(T)\big)\,dT \Big| \le  \frac{\lvert\AK\rvert\big(\frac53\lvert\AK\rvert  + \frac32Z\big)}{m}\,I_1,
\end{align}
with notation  as in   \eqref{idea}, $m:=[L:K]$ and $Z :=  \big(\e^{\lvert\AK\rvert k^4 D^4 m^{-1/3}}-1\big)/\big(\lvert\AK\rvert k^4 D^4 m^{-1/3} \big).$
\end{lemma}

\begin{proof}
 Lemma \ref{JamesLemma}, applied to the matrix $\mathcal{Q}$ and $a_w := \sqrt{m_w \alpha_{\kappa_w}''\!\big(S_w(\sigma)\big)}$, shows 
\begin{equation}\label{Lemmaapplied}
\sqrt{m_w  \alpha_{\kappa_w}''\!\big(S_w(\sigma)\big)}\,\lvert S_w(T)\rvert 
       \leq \sum_{w_0\in \eta_0} \sqrt{m_{w_0} \alpha_{\kappa_{w_0} }''\!\big(S_{w_0} (\sigma)\big)}\,\lvert S_{w_0}(T)\rvert\,
\end{equation}
for $ w\in\AK,\ T\in\R^k$ and $\eta_0$ as in  \eqref{eta0def}.
Since $x \mapsto x^4$ is convex, we have, 
\begin{align*}
m_w^2 \alpha_{\kappa_w}''\!\big(S_w(\sigma)\big)^2 S_w(T)^4 
      &\leq \Big(\sum_{w_0\in \eta_0} \sqrt{m_{w_0}  \alpha_{\kappa_{w_0} }''\!\big(S_{w_0} (\sigma)\big)}\,\lvert S_{w_0}(T)\rvert\Big)^4 \\
			&\leq k^3 \sum_{w_0\in \eta_0} m_{w_0}^2  \alpha_{\kappa_{w_0} }''\!\big(S_{w_0} (\sigma)\big)^2\, S_{w_0}(T)^4 .
\end{align*}
For $T\in\Delta$ and $w_0\in \eta_0$, by \eqref{deltaw} and \eqref{Deltadef} we have 
\[m_{w_0} \alpha_{\kappa_{w_0}}''\!\big(S_{w_0}(\sigma)\big)^2 S_{w_0}(T)^4 \leq 
  m_{w_0} \alpha_{\kappa_{w_0}}''\!\big(S_{w_0}(\sigma)\big)^2 \delta_{w_0}^4 = D^4 m_{w_0}^{-1/3}.\]
Hence,
\begin{align*}
m_w \alpha_{\kappa_w}''\!\big(S_w(\sigma)\big)^2 S_w(T)^4 &\leq k^3 \sum_{w_0\in \eta_0} \frac{m_{w_0}}{m_w}  \frac{ D^4}{m^{1/3}}\leq k^3 \sum_{w_0\in \eta_0}  \frac{2 D^4}{m^{1/3}}=\frac{2 k^4 D^4}{m^{1/3}}.
\end{align*}
 Combining this with Lemma \ref{Lem:Rho}, we conclude that for $T\in\Delta$,
\begin{align*}
  \big|\re \big(\rho(T) \big)  \big|
     =  \Big| \sum_{w\in\AK} m_w \re \big(\rho_w(T) \big) \Big|  
 \leq \sum_{w\in\AK} k^4 D^4 m^{-1/3}  =\lvert\AK\rvert k^4 D^4 m^{-1/3}.
\end{align*}
 Lemmas \ref{Lem:Rho} and \ref{Lem:uv} now show that for $T\in\Delta$,
\begin{align}\nonumber 
\big| &\re\big(\e^{\rho(T)}-1\big) \big|
\leq \frac{\im\big(\rho(T)\big)^2}{2} + \re\big(\rho(T)\big)Z \\ \nonumber 
 &\leq \frac{1}{2}\Big(\frac{\sqrt{2}}{3} \sum_{w\in\AK} m_w \alpha_{\kappa_w}''\!\big(S_w(\sigma)\big)^{3/2} |S_w(T)|^3\Big)^2 
          + \frac{Z}{2} \sum_{w\in\AK} m_w \alpha_{\kappa_w}''\!\big(S_w(\sigma)\big)^2 S_w(T)^4  \\ \label{realrhobound}
 &\leq \frac{\lvert \AK \rvert}{9}\sum_{w\in\AK} m_w^2  \alpha_{\kappa_w}''\!\big(S_w(\sigma)\big)^3 S_w(T)^6  
          +  \frac{Z}{2} \sum_{w\in\AK} m_w \alpha_{\kappa_w}''\!\big(S_w(\sigma)\big)^2 S_w(T)^4 ,
\end{align}
where in the last step we used the convexity of $x \mapsto x^2$.

By Lemma \ref{Lem:Rho}, $\im\big(e^{\rho(T)}\big)$ is  odd, while  $\re\big(e^{\rho(T)}\big)$ is even in $T$. Furthermore, $\M(T)$ is a real and even function  of $T$, and $\Delta$ is mapped to itself  by  $T \mapsto -T$. Hence, using \eqref{GT} and \eqref{realrhobound},
\begin{align*}
&\Big| \int_{\Delta} \big(\GT-\M(T)\big)  \,dT\Big|  =\Big|  \int_{\Delta} (e^{\rho(T)}-1)\M(T) \,dT\Big| =\Big| \int_{\Delta} \re(e^{\rho(T)}-1)\M(T) \,dT\Big|  \\
			 &\leq \sum_{w\in\AK} \int_{\R^k}  \Big(\frac{| \AK |}{9} m_w^2 \alpha_{\kappa_w}''\!\big(S_w(\sigma)\big)^3 S_w(T)^6 + \frac{Z}{2} m_w \alpha_{\kappa_w}''\!\big(S_w(\sigma)\big)^2 S_w(T)^4 \Big)  \M(T) \,dT.
\end{align*}
Using Lemma \ref{Lem:bw} and Corollary \ref{I1term}, we find 
\begin{align*}
\Big|\int_{\Delta} \big(\GT-\M(T)\big) \,dT\Big|
   & \leq \Big(\sum_{w\in\AK} \frac{\frac53\lvert\AK\rvert +  \frac32Z}{m_w}\Big) 
		                \frac{(2\pi)^{k/2}\e^{ \alpha (\sigma)-ny\cdot\sigma}}
			  						     {\sqrt{\det(H(\sigma))}}\\ &
\le\frac{ \lvert\AK\rvert\big(\frac53\lvert\AK\rvert  + \frac32Z\big)}{m}I_1.\qedhere
\end{align*}
\end{proof}
 Our next estimate will let us deal with the term $ \int_\E \|ax\|^2 \e^{- t\,\|ax\|^2}\,d\mu(x)$ in the Basic Inequality \eqref{MainIneq} and \eqref{derivpiece}.

\begin{lemma}\label{derivterm} For $y\in\R^k$ and $m:=[L:K]\ge1000$  we have 
\begin{equation}\label{DerivEst}
 \int_{T\in\R^k}  \big|T_1\e^{\alpha(\sigma+iT)-ny\cdot (\sigma+iT) }\big|dT   \le\frac{1.66 \cdot 1.0021^{k-1}k\sqrt{ \lvert\AKk\rvert }}{\sqrt{m}} \sigma_1 I_1 ,
\end{equation}
with $I_1$  as in \eqref{idea}, $\alpha$ as in  \eqref{Walpha} and $\sigma=(\sigma_1,\ldots,\sigma_k):=\sigma(ny)$ as  in Lemma \ref{Criticalpoint}.
\end{lemma}
\begin{proof}
By \eqref{nice}, for $T\in\R^k$ we have 
\begin{equation}\label{nT1}
nT_1= \sum_{v\in\AL}e_vS_v(T)
    =\sum_{w\in\AK}\sum_{v|w} e_v S_w(T) 
		=\sum_{w\in\AK} m_w S_w(T).
\end{equation}
Hence we will need to bound integrals of the kind $\int_{\R^k}|S_w(T) \e^{\alpha(\sigma+iT) }|\,dT$.

 Let $\eta_0$ be as in \eqref{eta0def} and let $w_0 \in\eta_0$. Then,  using \eqref{gammaineq} and changing variables as in the proof of Lemma \ref{I2I3terms}, 
\begin{align*}
&\int_{\R^k} \big| S_{w_0}(T) \e^{\alpha(\sigma+iT)-\alpha(\sigma)} \big|\,dT 
\le \int_{\R^k}\Big| S_{w_0}(T) \prod_{w\in\eta_0} \e^{m_w \alpha_{\kappa_w}(S_w(\sigma+iT))-m_w \alpha_{\kappa_w}(S_w(\sigma))} \Big| \,dT \\
	 &= \frac{1}{\lvert \det(\mathcal{Q}_{ \eta_0})\rvert} 
	         \int_{-\infty}^\infty \lvert \tilde{T}_{w_0} 
					                             \e^{m_{w_0} \alpha_{\kappa_{w_0}}(S_{w_0}(\sigma)+i\tilde{T}_{w_0}) 
																			   -m_{w_0} \alpha_{\kappa_{w_0}}(S_{w_0}(\sigma))} 
																 \rvert \,d\tilde{T}_{w_0}    \\
	 &\quad\quad\quad
				\times\prod_{\substack{w \in\eta_0 \\ w \neq w_0}}
				   \int_{-\infty}^\infty \lvert \e^{m_w \alpha_{\kappa_w}(S_w(\sigma)+i\tilde{T}_w)-m_w \alpha_{\kappa_w}(S_w(\sigma))} \rvert 
					                    \,d\tilde{T}_w.
\end{align*}
Using Lemma \ref{Lem:A}  and \eqref{Comparison}  we obtain,
\begin{align*}
&\int_{\R^k}\big| S_{w_0}(T) \e^{\alpha(\sigma+iT)-\alpha(\sigma)}\big| \,dT \le \frac{1}{| \det(\mathcal{Q}_{\eta_0})|} 
	       \frac{0.83\sqrt{2\pi}}    {m_{w_0}\alpha_{\kappa_{w_0}}''\!\big(S_{w_0}(\sigma)\big)}
				 \prod_{\substack{w \in\eta_0 \\ w \neq w_0}}
				             \frac{1.0021 \sqrt{2\pi}} {\sqrt{m_w\alpha_{\kappa_w}''\!\big(S_w(\sigma)\big)}} \\
	 &= \frac{0.83\cdot 1.0021^{k-1}}{\sqrt{m_{w_0} \alpha_{\kappa_{w_0}}''\!\big(S_{w_0}(\sigma)\big)}} 
	    \frac{(2\pi)^{k/2}}
			     {\big| \det(\mathcal{Q}_{\eta_0})\big| \prod_{w\in \eta_0} \sqrt{m_w \alpha_{\kappa_w}''\!\big(S_w(\sigma)\big)}} \\
	 &\leq 
	       \frac{0.83\cdot 1.0021^{k-1}\sqrt{ \lvert\AKk\rvert }}{\sqrt{m \alpha_{\kappa_{w_0}}''\!\big(S_{w_0}(\sigma)\big)}}
	       \frac{(2\pi)^{k/2}}
					    {\sqrt{\det\!\big(H(\sigma)\big)}}.
\end{align*}
By inequality \eqref{Lemmaapplied},
\begin{align*}
\sum_{w\in\AK} m_w \lvert S_w(T)\rvert 
           &= \sum_{w\in\AK} \sqrt{\frac{m_w}{ \alpha_{\kappa_w}''\!\big(S_w(\sigma)\big)}} 
					                  \sqrt{m_w  \alpha_{\kappa_w}''\!\big(S_w(\sigma)\big)} \,\lvert S_w(T)\rvert \\
				&\leq \sum_{w\in\AK} \sqrt{\frac{m_w}{ \alpha_{\kappa_w}''\!\big(S_w(\sigma)\big)}} 
				      \sum_{w_0\in\eta_0} \sqrt{m_{w_0}\alpha_{\kappa_{w_0}}''\!\big(S_{w_0}(\sigma)\big)}\, \lvert S_{w_0}(T)\rvert \\
				&\leq 2\sum_{w\in\AK} m_w S_w(\sigma)
				      \sum_{w_0\in \eta_0} \sqrt{\alpha_{\kappa_{w_0}}''\!\big(S_{w_0}(\sigma)\big)}\, \lvert  S_{w_0}(T)\rvert,
\end{align*}
where the last inequality uses $m_{w_0} \leq 2m_w$ and $x^2 \alpha_{\kappa_w}''(x) > \kappa_w \geq 1/2$ for $x>0$  \cite[(5.7)]{FS}. Hence, by \eqref{nT1},
\[\sum_{w\in\AK} m_w \lvert S_w(T)\rvert \leq 2n\sigma_1  \sum_{w_0\in \eta_0} \sqrt{\alpha_{\kappa_{w_0}}''\!\big(S_{w_0}(\sigma)\big)}\, \lvert  S_{w_0}(T)\rvert.\]
It follows that
\begin{align*}
\int_{T\in\R^k} &\big|T_1\e^{\alpha(\sigma+iT)-ny\cdot (\sigma+iT) }\big|\,dT =\frac{\e^{-ny\cdot \sigma}}{n}\int_{\R^k} \Big|\Big(\sum_{w\in\AK} m_w S_w(T)\Big)e^{\alpha(\sigma+iT)}\Big| \,dT \\
            &\leq \frac{\e^{\alpha (\sigma)-ny\cdot \sigma}}{n } \cdot 2n\sigma_1
						                \sum_{w_0\in \eta_0} \sqrt{\alpha_{\kappa_{w_0}}''(S_{w_0}(\sigma))} \int_{\R^k} \big| S_{w_0}(T) \e^{\alpha(\sigma+iT)-\alpha (\sigma)} \big| \,dT \\
						&\leq 2\sigma_1 \sum_{w_0\in \eta_0} \frac{0.83\cdot 1.0021^{k-1}\sqrt{ \lvert\AKk\rvert }}{\sqrt{m }}			
	       \frac{(2\pi)^{k/2}\e^{\alpha(\sigma)-ny\cdot \sigma}}
					    {\sqrt{\det\!\big(H(\sigma)\big)}}
 \\
						&=\frac{1.66 \cdot 1.0021^{k-1}k\sqrt{ \lvert\AKk\rvert }}{\sqrt{m}} \sigma_1 I_1,
\end{align*}
where the last equality uses Corollary \ref{I1term}.
\end{proof}

\section{Proof of the Main Theorem}\label{ProofMainTheorem} 
The next lemma will allow us to ensure that each integral  in the Basic Inequality \eqref{MainIneq} is positive. As in \S\ref{Asymptotics}, we always assume that $\ELK\subset E \subset \OL$.  
\begin{lemma}\label{positerms}There is an absolute constant $N_0$ such that
if $[L:K]\ge N_0 \cdot 2.01^{[K:\Q]}$ and  $a\in\OL,\ a\not=0$,   then  for $t:=\exp\!\big(\Psi(0.51+\frac{r_2}{2n})\big)$ we have $\sigma_1(ny_{a,t})\ge0.51$ and 
$$
\int_{x\in\E}\Big(\frac{2t\|ax\|^2}{n}-1\Big)\e^{- t\,\|ax\|^2}\,d\mu_\E(x)>0.01I_1(ny_{a,t})  \mathcal{L} ,
$$
where $y_{a,t}$ is given by Corollary \ref{sigma1ineq}, $\Psi(x):=\Gamma^\prime(x)/ \Gamma(x)$, and 
$$
 \mathcal{L}= \frac{\sqrt{\det(\mathcal{Q}_{\phantom{l}}^\intercal \mathcal{Q})\prod_{w\in\AK}(r_{1,w}+r_{2,w})}}{2^{r_1}(2\sqrt{\pi})^{r_2}\pi^k} , 
\ \,
 I_1(ny)= \frac{(2\pi)^{k/2} \e^{\alpha(\sigma)-ny\cdot\sigma}}
                               {\sqrt{\det\!\big(H(\sigma)\big)}},\ \, \sigma:=\sigma(ny_{a,t}) .
$$
\end{lemma}
 \begin{proof}
We note that $ \mathcal{L}$ is as in Corollary  \ref{sigma1ineq}, except that we used \eqref{QcalQ} to express $ \mathcal{L}$  in terms of
$\mathcal{Q}$ rather than $Q$. Letting  $ y:=y_{a,t}$, from Corollary \ref{sigma1ineq} we have 
\begin{align} 
&\frac{\int_{\E}\big(\frac{2t\|ax\|^2}{n}-1\big)\e^{- t\,\|ax\|^2}\,d\mu_\E}{\int_{\E}\e^{- t\,\|ax\|^2}\,d\mu_\E}
=\frac{    \int_{T\in\R ^k}(2(\sigma_1+iT_1)-1)   \,\e^{\alpha (\sigma+iT)-ny\cdot(\sigma+iT)}  \,dT}
{  \int_{T\in\R ^k} \,\e^{\alpha (\sigma+iT)-ny\cdot(\sigma+iT)}    \,dT} \nonumber\\  \label{startineq}
&\quad=2\sigma_1-1+\frac{2i^{1-k}  \int_{\R ^k}T_1   \,\e^{\alpha (\sigma+iT)-ny\cdot(\sigma+iT)}  \,dT}{i^{-k} \int_{\R ^k} \e^{\alpha (\sigma+iT)-ny\cdot(\sigma+iT)}  \,dT}.
\end{align}   
Again  from   Corollary \ref{sigma1ineq}, for $a\in\OL,\ a\not=0$, 
 \begin{equation}\label{yat1}
y_1:=(y_{a,t})_1= \log(t) + \tfrac{2}{n}\log|\mathrm{Norm}_{L/\Q}(a)| \ge \log(t)= \Psi(0.51+\tfrac{r_2}{2n})  .
\end{equation}
Applying Lemma \ref{any field} to $ny$, since $\Psi^{-1}$ is increasing   we have, 
\begin{align}  \label{SigmaOneEst}  
\sigma_1=\sigma_1(ny_{a,t}) \ge  \digamma^{-1}(y_1 )-\frac{r_2}{2n}\ge \digamma^{-1}\big( \Psi(0.51+\tfrac{r_2}{2n})  \big)-\frac{r_2}{2n} = 0.51.
\end{align}

Since $k\le|\AK|\le[K:\Q]$ by \eqref{kBound}, we have $|\AKk|=\binom{|\AK|}{k}\le 2^{|\AK|}\le 2^{[K:\Q]}$. Thus,  Lemma \ref{derivterm} yields
\begin{align}\nonumber
& 2\int_{T\in\R^k}  \big|T_1\e^{\alpha(\sigma+iT)-ny\cdot (\sigma+iT) }\big|dT   \le\frac{3.32 \cdot 1.0021^{k-1}k\sqrt{ \lvert\AKk\rvert }}{\sqrt{m}} \sigma_1 I_1(ny)\\
&\ \    \le\frac{3.32 \cdot 1.0021^{[K:\Q]}[K:\Q] 2^{[K:\Q]/2}}{\sqrt{m}} \sigma_1 I_1(ny)<0.01\sigma_1 I_1(ny)  \label{badDeriv}
\end{align}
for $m\ge N_0 \cdot 2.01^{[K:\Q]}$ and some absolute $N_0\ge500$. By \eqref{idea} and \eqref{fintegral} we have
$$
\frac{1}{i^k}\int_{\R ^k} \e^{\alpha (\sigma+iT)-ny\cdot(\sigma+iT)}\,dT= I_1+ I_2- I_3+ I_4,
$$
where $I_j=I_j(ny)$.
Taking $D=1$ in Lemmas \ref{I2I3terms} and  \ref{I4term}, and after possibly enlarging $N_0$, we obtain
$|I_2|+|I_3|+|I_4|\le0.01I_1.$ Hence,  
\begin{equation}\label{goodbound}
 \frac{1}{i^k}\int_{\R ^k} \e^{\alpha (\sigma+iT)-ny\cdot(\sigma+iT)}\,dT\ge0.99I_1,
\end{equation}
and so, since   $\sigma_1\ge0.51$ by \eqref{SigmaOneEst},  
$$
2\sigma_1-1+\frac{2 i^{1-k}  \int_{\R ^k}T_1   \,\e^{\alpha (\sigma+iT)-ny\cdot(\sigma+iT)}  \,dT}{i^{-k} \int_{\R ^k} \e^{\alpha (\sigma+iT)-ny\cdot(\sigma+iT)}  \,dT} \ge 2\sigma_1-1-\frac{0.01\sigma_1}{0.99}>1.989\sigma_1-1>0.014.
$$
 A glance at \eqref{startineq}  shows that we are finished.
\end{proof}

We now prove the Main Theorem in \S1, which we do not repeat here.
Note  that 
\begin{equation}\label{volume}
 \| \varpsilon_1\wedge\cdots\wedge\varepsilon_j\|_2=\mu_\E(\E/E)\ge \frac{\mu_\E(\E/E)}{\sizeofEtor}.
\end{equation}
Take $N_0$ and  $t:=\exp\!\big(\Psi(0.51+\tfrac{r_2}{2n})\big) $ as in Lemma \ref{positerms}. In the Basic Inequality \eqref{MainIneq}  take $\a:=\OL$, so that the sum there
includes only nonzero $a\in\OL$. By Lemma \ref{positerms}, each integral  in the sum is positive. Retaining only the term corresponding to $a=1\in\OL$ we have, again by Lemma \ref{positerms},  
\begin{align}\label{bound1}
\frac{\mu_\E(\E/E)}{\sizeofEtor}>0.01 \frac{2^{k/2} \sqrt{\det(\mathcal{Q}_{\phantom{l}}^\intercal \mathcal{Q})\prod_{w\in\AK}(r_{1,w}+r_{2,w})}}{\sqrt{\det\!\big(H(\sigma)\big)}\pi^{k/2} } \frac{(2/\sqrt{\pi})^{r_2} \e^{\alpha(\sigma)-ny\cdot\sigma}}
                               {2^n}\,
\end{align}
where $y:=y_{1,t}$ and $\sigma:=\sigma(ny)$.  Corollary \ref{sigma1ineq} applied to $a=1$ gives 
\begin{equation}\label{y1tnailed}
y =(\log(t),0,0,\ldots,0)=(\Psi(0.51+\tfrac{r_2}{2n}),0,\ldots,0).
\end{equation}

We need an upper bound for $\det\!\big(H(\sigma)\big)$ in \eqref{bound1}. In view of  \eqref{Eqn:detH}, we look for 
an upper bound for $ \alpha_{\kappa_w}''\!\big(S_w(\sigma)\big).$ Note that
$$
 \alpha''_\kappa(x)=\kappa\Psi'(x)+(1-\kappa)\Psi'(x+\tfrac12)\le\Psi'(x)\qquad\qquad\qquad(0\le\kappa\le1,\ x>0),
$$
since $\Psi'(x)$ is decreasing for $x>0$. Note 
that  $\sigma_1\ge0.51$ by \eqref{SigmaOneEst} and that 
\begin{equation}\label{goodbound2}
-2<\Psi(0.51)\le y_1=\Psi(0.51+\tfrac{r_2}{2n})\le \Psi(0.76)<-1.
\end{equation}
From Lemma \ref{BoundSvl} we have
\begin{align} \nonumber
S_w(\sigma) &\ge \frac{1}{    (n-1)\log(2\sigma_1+\tfrac12) - ny_1} 
\ge  \frac{1}{   n(\log(3\sigma_1)+2)} >  \frac{1}{   n\log(23\sigma_1) }.
\end{align} 
Estimating the series by an integral,
 $
\Psi'(x)=\sum_{k=0}^\infty\frac{1}{(k+x)^2}<\frac{1}{x}+\frac{1}{x^2},
$
yields
$$
 \alpha_{\kappa_w}''\!\big(S_w(\sigma)\big)<\Psi'\!\big(S_w(\sigma)\big)<\frac{1}{S_w(\sigma)}+\frac{1}{\big(S_w(\sigma)\big)^2}< 2  n^2\log^2(23\sigma_1).
$$
From  $\det(\mathcal{Q}_{\phantom{l}}^\intercal \mathcal{Q})=\sum_{\eta\in\AKk}\det^2(\mathcal{Q}_\eta)$ (Cauchy-Binet),  $r_{1,w}+r_{2,w}\ge m_w/2$ and   \eqref{Eqn:detH}, 
\begin{align}\label{bound3}
 \frac{2^{k/2} \sqrt{\det(\mathcal{Q}_{\phantom{l}}^\intercal \mathcal{Q})\prod_{w\in\AK}(r_{1,w}+r_{2,w})}}{\sqrt{\det\!\big(H(\sigma)\big)}\pi^{k/2} } \ge\Big(\frac{1}{ \sqrt{2\pi }\,n \log (23\sigma_1)}\Big)^{[K:\Q]},
 \end{align}
where we also used $k\le|\AK|\le[K:\Q]$.

We now bound the term $\e^{\alpha(\sigma)-ny\cdot\sigma}$ in \eqref{bound1} from below. 
From \eqref{y1tnailed} and \eqref{goodbound2}, 
$$
-ny\cdot\sigma=-n\sigma_1y_1 > n\sigma_1.
$$
Using the lower bound for $\alpha(\sigma)$ in Lemma \ref{alphaineq}, we have
\begin{equation}\label{bound4}
\alpha(\sigma)-ny\cdot\sigma \ge  n \log \Gamma\big(\sigma_1+\textstyle\frac{r_2}{2n}\big)-n\sigma_1y_1.
\end{equation}
We now distinguish two cases according to the size of $\sigma_1$. If $\sigma_1\ge4$, then $\log \Gamma\big(\sigma_1+\textstyle\frac{r_2}{2n}\big)\ge\log(6)$. Since $-n\sigma_1y_1>n\sigma_1$, after possibly increasing $N_0$, the Main Theorem follows easily from \eqref{volume}, \eqref{bound1}, \eqref{bound3} and \eqref{bound4}.

We now turn to the remaining case, \ie $0.51\le\sigma_1< 4$. (By Lemma \ref{positerms}, $\sigma_1\ge0.51$.) Then in \eqref{bound3} we can replace $\log(23\sigma_1)$ by 5.
 The critical points $r\in(0,\infty)$ of $r \mapsto \log \Gamma\big(r+\textstyle\frac{r_2}{2n}\big)- ry_1$ occur  where 
$$
\Psi\big(r+\textstyle\frac{r_2}{2n}\big)=y_1:=\Psi(0.51+\tfrac{r_2}{2n}).
$$
But  $\Psi\colon (0,\infty)\to\R$ is injective, so $r=0.51$ is the only critical point of $r \mapsto \log \Gamma\big(r+\textstyle\frac{r_2}{2n}\big)- ry_1$ , and it is a local (therefore global) minimum. Since $\sigma_1\ge0.51$,
$$
\alpha(\sigma)-ny\cdot\sigma \ge  n \big(\log \Gamma\big(0.51+\textstyle\frac{r_2}{2n}\big)- 0.51y_1\big)=n  \big(\log \Gamma\big(0.51+\textstyle\frac{r_2}{2n}\big)- 0.51\Psi(0.51+\tfrac{r_2}{2n})\big).
$$
Note that $0\le\textstyle\frac{r_2}{2n}\le\frac14$, $\Psi(r)<-1$ for $0<r<0.76$, and $\Psi^\prime(r)>0$ for $r>0$. Hence 
$$
x \mapsto \log \Gamma(0.51+x)- 0.51\Psi(0.51+x)+x\log(4/\pi)
$$
 is decreasing for $0\le x\le\frac14$.  We conclude that
\begin{align*}
\alpha&(\sigma)-ny\cdot\sigma +r_2\log(2/\sqrt{\pi})-n\log(2) 
\\ &
\ge n  \big(\log\Gamma(0.76)- 0.51\Psi(0.76)+0.25\log(4/\pi)-\log(2)\big)
>    n/10.
\end{align*}
Since $\e^{0.0955}>1.1$ and $j:=\text{rank}_\Z(E)\le|\AL|\le n$, after again possibly increasing $N_0$,  we can use the ``spare" $\exp(0.0045n)$ to control the term in  \eqref{bound3}. $\square$

We note that the our proof of the Main Theorem shows that the  $1.1^j$ appearing in it can be replaced by $\exp\!\big(nf(r_2/(2n))\big)$, where $r_2$ is the number of complex places of $L$ and 
$$f(x):=\log\Gamma(0.51+x)- 0.51\Psi(0.51+x)+x\log(4/\pi)-\log(2).$$ In particular, if $L$ is totally real, we can replace $1.1^j$ by $2.3^n$. We can also replace $0.51$ above by $\epsilon+1/2$ for any $\epsilon>0$. 

Finally, we prove that  every element of    $\bigwedge^{r_L-1}\LOG(\OLU)$ is represented by a pure wedge, as claimed in the Introduction.

\begin{lemma} \label{s:Codim1} Suppose  $M$ is a $\Z$-lattice in $\R^n$ of rank $n \ge 1$. Then every element of $w\in\bigwedge^{n-1} M$ has the form 
$$\omega= d\epsilon_1 \wedge \epsilon_2 \wedge \cdots \wedge \epsilon_{n-1} $$
for some integer $d$ and some basis $\{\epsilon_1,\ldots,\epsilon_n\}$ of $M$
as a $\Z$-module.
\end{lemma}

\begin{proof} We may clearly assume  $  \omega\not=0$. Define the homomorphism $\wedge_\omega:M\to \bigwedge^n M$ by
$\wedge_\omega(m):=\omega\wedge m$. As $\bigwedge^n M\cong\Z$, $M/\ker(\wedge_\omega)$ is torsion-free and so $\ker(\wedge_\omega)$ is a direct summand of $M$ of rank $n-1$. Let  $\epsilon_1,...,\epsilon_n$ be a $\Z$-basis of $M$ such that $\epsilon_1,...,\epsilon_{n-1}$ is  a $\Z$-basis of $\ker(\wedge_\omega)$, let $\eta:=\epsilon_1\wedge\cdots\wedge \epsilon_{n-1}\in \bigwedge^{n-1} M$, and define $d\in\Z$ by $\omega\wedge\epsilon_n=d\eta\wedge \epsilon_n$. 
Notice that $\eta\wedge\epsilon_i=0=\omega\wedge\epsilon_i$ for $1\le i\le n-1$. 

For $m\in M$, write $m=\sum_{i=1}^n a_i\epsilon_i$ with $a_i\in\Z$. Then 
$$ 
\omega\wedge m=\omega\wedge \sum_{i=1}^n a_i\epsilon_i=a_n\omega\wedge\epsilon_n=a_n d \eta\wedge\epsilon_n=d \eta\wedge\sum_{i=1}^n a_i\epsilon_i=d \eta\wedge  m.
$$ 
As the $\wedge$-pairing of $\bigwedge^{n-1} M$ with $M$ is non-degenerate,   $\omega=d\eta=d\epsilon_1\wedge\cdots\wedge\epsilon_{n-1}$. 
\end{proof}

 \section{Appendix by Fernando Rodriguez Villegas (May 2002)\\ Some remarks on Lehmer's conjecture}\label{RVappendix}
\subsection{} 
The {\it logarithmic Mahler measure} of a non-zero 
Laurent polynomial  $P \in $\\ $\C 
[x_1^{\pm1}, \ldots , \;  x_n^{\pm1}]$
is defined  as
\begin{equation}\label{tag1}
m(P) = \int_0^1 \cdots \int_0^1 \, \log \left| P(e^{2\pi i \theta_1}
\, , \ldots , \; 
      e^{2\pi i\theta_n}) \right| d\theta_1 \cdots d\theta_n \;
\end{equation}
and its {\it Mahler measure}  as $M(P)=e^{m(P)}$, 
 the geometric mean of $|P|$ on the torus 
$$
T^n=\big\{ (z_1,
\ldots , z_n) \in \C^n \big| \; |z_1|= \ldots = |z_n| = 1 \big\}.
$$

When $n=1$ Jensen's formula gives the identity
\begin{equation}\label{tag3}
M(P)= |a_0|\prod_{|\alpha_\nu|>1}|\alpha_\nu| \; ,
\end{equation}
where $P(x)=a_0\prod_{\nu=1}^d(x-\alpha_\nu)$, from which we clearly
obtain that $M(P)\geq 1$ if $P\in \Z[x]$. By a  theorem of
Kronecker  if $M(P)=1$ for $P\in \Z[x]$ then $P$ is {\it cyclotomic},
i.e., $P$ is monic and its roots are either $0$ or roots of unity.

In the early 30's Lehmer \cite{Le} famously asked whether there is 
an absolute lower bound for $M(P)$ when $P\in \Z[x]$ and $M(P)>1$. 
The purpose of this note is to point out a simple reformulation of
this  question in terms of the logarithmic embedding of units of a
number field and, given this setting, to propose a natural
generalization. 
\subsection{} We start with some general observations about $m(P)$. First
of all, the fact that the integral in \eqref{tag1} is finite for all non-zero
$P$ does need a proof. Here is a sketch. Using Jensen's
formula we find,  as in \eqref{tag3} that
\begin{equation}\label{tag4}
m(P) = m(a_0) + \frac{1}{(2\pi i)^n} \sum_{\nu=1}^d \int_{T^{n-1}}
\log^+|\alpha_\nu(y)|  \frac{dy}{y},
\end{equation}
where $y=(y_1,\cdots,y_{n-1})$, $dy/y=dy_1/y_1\cdots
dy_{n-1}/y_{n-1}$, $\log^+(x)=\max\{\log|x|,0\}$, and
$a_0(y),\alpha_v(y),d$ are the leading coefficient, roots and degree, 
respectively, of $P$ viewed as a polynomial in $x_n$. The
$\alpha_\nu$'s are algebraic functions of $y \in \C^{n-1}$, 
continuous and piecewise smooth, except at those $y$'s where $a_0(y)$
vanishes (where some will go off to infinity). 

We can apply the above procedure to any variable $x_n$ on the torus
$T^n$. It is not hard to see that we may change coordinates in such a
way that $a_0(y)$ is actually constant, completing the proof by
induction on $n$.

This last remark can be expanded. Let $\Delta$ be the Newton polytope
of $P$; i.e., the convex hull of the exponents $m\in \Z^n$ of
monomials $x^m=x_1^{m_1}\cdots x_n^{m_n}$ 
such that if
$$
P=\sum_{m\in \Z^n} c_mx^m\,,
$$
then $c_m\neq0$.

We define a {\it face} $\tau$ of $\Delta$ as the non-empty
intersection of $\Delta$ with a half-space in $\R^n$.  Chose a
parameterization $\phi: \R^k \longrightarrow \R^n$ of the affine
subspace of smallest dimension containing $\tau$; $k$ is the {\it
dimension} of the face $\tau$. Define
$$
P_\tau= \sum_{m\in \Z^k} c_{\phi(m)}x^m\; ,
$$
a polynomial whose own Newton polytope is $\phi^{-1}(\tau)$. We call
$P_\tau$ the {\it face polynomial} associated to the face $\tau$. It
depends on a choice of $\phi$ but note that by changing variables in
the integral $m(P_\tau)$ is actually independent of that choice.

It is not hard to see that for any facet (co-dimension $1$ face) $\tau
 \subset \Delta$ we can choose $\phi$ and system of coordinates in
 $T^n$ so that, in the notation of \eqref{tag4}, $P_\tau= a_0(y)$. By \eqref{tag4} and
 induction on $n$ we conclude \cite{Sm1} that
\begin{equation}\label{tag7}
m(P_\tau)\leq m(P), \qquad \text{for all faces} \quad \tau \subset
\Delta\,.
\end{equation}
In particular, 
$$
m(P)\geq 0, \qquad \text{for} \quad 0\neq P \in
\Z[x_1,x_1^{-1},\ldots, x_n,x_n^{-1}]\,. 
$$
Also, since clearly $m(PQ)=m(P)+m(Q)$, we have that
\begin{equation}\label{tag9}
m(Q)\leq m(P), \qquad \text{if} \quad Q\mid P, \quad  
0\neq P,Q \in \Z[x_1,x_1^{-1},\ldots, x_n,x_n^{-1}]\,. 
\end{equation}
Though Lehmer's conjecture is about polynomials in one variable,
polynomials in more variables are also relevant due to the following
result \cite{Bo}. For any $0\neq P \in \Z[x_1,x_1^{-1},\ldots, x_n,x_n^{-1}]$
and $0\neq(a_1,\ldots,a_n)\in \Z^n$ we have
\begin{equation}\label{tag10}
\lim_{k\rightarrow \infty} m(Q_k)=m(P)\qquad \text{where} \quad
Q_k(t)=P(t^{a_1k},\ldots,t^{a_nk})
\end{equation}
That is, there are one variable polynomials $Q$ with $m(Q)$ as close
to $m(P)$ as desired. (We should note that \eqref{tag10} is not an immediate
consequence of general results about integration but requires a 
somewhat delicate analysis.)

\subsection{}Let us go back to polynomials in one variable. If we want to
find polynomials $P\in \Z[x]$ with positive but small $m(P)$, by \eqref{tag7}
and \eqref{tag9} (and Gauss' lemma) we may as well restrict ourselves to minimal 
polynomials of algebraic units.

Let $F$ be a number field of degree $n$. Let $I$ be the set of
embeddings $\sigma: F\longrightarrow \C$ and $V$ the real vector space
of formal linear combinations
$$
\sum_{\sigma \in I} \alpha_\sigma [\sigma], \qquad \alpha_\sigma \in
\R\,.
$$
We have the decomposition
$$
V=V^+\oplus V^-\,,
$$
where $V^{\pm}$ is the subspace of $V$ where  complex conjugation acts
like $\pm1$. We let $n_\pm = \dim_\R V^\pm$ (in terms of the standard
notation $n_+=r_1+r_2$ and $n_-=r_2$).

By Dirichlet's theorem the image of the unit group $\mathcal{O}_F^*$ by the
log map
$$
\begin{matrix}
l_1:&\mathcal{O}_F^* & \longrightarrow & V\\
&\epsilon & \mapsto & \sum_{\sigma \in I} \log |\epsilon^\sigma|\,
[\sigma] 
\end{matrix}
$$
is a discrete subgroup $L_1\subset V$ of rank $r=n^+-1$.

On $V$ we define the $L^1$-norm
$$
\Big\|\sum_{\sigma \in I} \alpha_\sigma [\sigma]\Big\|_1
:= \sum_{\sigma \in 
I}|\alpha_\sigma| 
$$
and we let
$$
\mu_{1,1}(F):=\min_{ l\in L_1\setminus \{0\}} \left|\left|l\right|\right|_1
$$
(the reason for this indexing will become clear shortly).

For any unit $\epsilon \in \mathcal{O}_F^*$ we have $|\N_{F/\Q}(\epsilon)|=1$
hence
\begin{equation}\label{tag16}
 \sum_{\sigma \in 
I}\log|\epsilon^\sigma|=0\,.
\end{equation}
Let $P \in \Z[x]$  be the (monic) minimal polynomial of $\epsilon$ and
$$
\left\| l_1(\epsilon)\right\|_1=\frac{2n}{n_\epsilon}\,m(P)\,.
$$
This simple observation allows us to reformulate Lehmer's conjecture
as follows.
\begin{conjecture*}(Lehmer) There exists an absolute constant
$\delta_1>0$ such that
\begin{equation}\label{tag18}
\mu_{1,1}(F)\geq \delta_1, \qquad \text{for all number fields $F$ with
$r\geq 1 $.} 
\end{equation}
\end{conjecture*}
   
\subsection{}\label{L1RV} Let $V$ be a vector space over $\R$ of dimension $n$ and
$L\subset V$ a discrete subgroup of rank $r \geq1$. A choice of basis
$v_1,\ldots,v_n$ for $V$ determines $L^1$-norms on $\Lambda^k V$ for
$k=1,\ldots, n$ by
$$
\Big\|\sum_{1\leq j_1<\cdots<j_k\leq n}
a_{j_1,\ldots,j_k}v_{j_1}\wedge\cdots\wedge v_{j_k}\Big\|_1:= 
\sum_{1\leq j_1<\cdots<j_k\leq n}
|a_{j_1,\ldots,j_k}|\,.
$$
For each $1\leq k\leq r$ we define (with respect to the chosen basis) 
$$
\mu_k(L) :=\min \left\|l_1\wedge\cdots\wedge l_k \right\|_1\,,
$$
where the minimum is taken over all $l_1,\ldots,l_k \in L$ which are
linearly independent over $\R$.

If $A$ is the $n\times k$ integral matrix whose $i$-th column
consists of the coordinates of $l_i$ in the basis $v_1,\ldots,v_n$
then, as it is easily seen,
$$
 \left\|l_1\wedge\cdots\wedge l_k \right\|_1= 
\sum_{A'} |\!\det A'|\,,
$$
where $A'$ runs over all $k\times k$ minors of $A$.

Returning to the number field situation of the previous section we
define the invariants
$$
\mu_{1,k}(F):=\mu_k(L_1)\,,
$$
where, as before, $L_1$ is the image of the units of $F$ under the log
map. 

A general version of Lehmer's conjecture would then be
\begin{conjecture*} For each $k\in \N$ there exists  an absolute constant
$\delta_k>0$ such that
$$
\mu_{1,k}(F)\geq \delta_k, \qquad \text{for all number fields $F$
with $ r\geq k$}.
$$
\end{conjecture*}
A straightforward calculation shows that the top invariant
$\mu_{1,r}(F)$, with $r=n^+-1$ the rank of the unit group $\mathcal{O}_F^*$,
equals the regulator of $F$. It is known \cite{Zi},  \cite{Fr},  \cite{Sk} that the
regulator of number fields is universally bounded below and hence the
above conjecture is true for $k=r$.

In summary, we have seen (18) that Lehmer's conjecture can be phrased
in terms of the $L^1$-norm of units under the log map. The above
conjecture is an attempt to quantify, in what seems to be the most
natural way, the question of what is the general shape of $L_1$, the
discrete group of units under the log map.

\subsection{}  We may carry these ideas a little further still. Borel
 proved, generalizing Dirichlet's result for units, that for each
$j>1$ there is a regulator map $\reg_j$
\begin{equation}\label{tag24}
\begin{matrix}
l_j:&K_{2j-1}(F) & \longrightarrow & V\\
&\xi & \mapsto & \sum_{\sigma \in I} \reg_j(\xi^\sigma)\,
[\sigma] 
\end{matrix}
\end{equation}
whose image is a discrete subgroup $L_j$ of $V^\pm$, with $\pm
=(-1)^{j-1}$, of rank $n^\pm$ and covolume related to the value of the
zeta function $\zeta_F$ of  $F$ at $s=j$. Here $K_{2j-1}(F)$
are the $K$ groups defined by Quillen.

We now define 
$$
\mu_{j,k}(F):=\mu_k(L_j), \qquad \text{for $1\leq k \leq n_\pm$}\,,
$$
and we may ask: what is the nature of these invariants, how
do they depend on the field $F$? Does the analogue of Lehmer's
conjecture hold?

Apart from their formal analogy with Lehmer's question, answers to
such questions can be quite useful in practice as we now illustrate.

\subsection{}  For general $j$, very little is known about the groups
$K_{2j-1}(F)$ or the map $\reg_j$. For $j=2$, however, things can be
made quite explicit (and of course $j=1$ corresponds to the case 
of units). Indeed, up to torsion, $K_3(F)$ is isomorphic to the {\it 
Bloch group} $\B(F)$,  defined by generators and relations as follows.   

For any field $F$ define
$$
\A(F):= \Big\{ \sum_i n_i [z_i] \in \Z[F] \;|\; \sum_i n_i (z_i\wedge
(1-z_i)) = 
0\Big\}\,,
$$
where the corresponding term in the sum is omitted if $z_i=0,1$ and 
$$
\mathcal{C}(F):=\bigg\{[x]+[y ]+\Big[\frac{1-x}{1-xy}\Big]+
[1-xy]+\Big[\frac{1-y}{1-xy}\Big] 
\;\big|\; x,y \in F,\; xy \neq 1\bigg\}\;.
$$
It is not hard to check that $\mathcal{C}(F)\subset \A(F)$. Finally, let
$$
\B(F):=\A(F)/\mathcal{C}(F)\;.
$$

We recall  the definition of the {\it Bloch--Wigner dilogarithm}. 
Starting with the usual dilogarithm
$$
\Li_2(z)=\sum_{n=1}^\infty \frac{z^n}{n^2}, \qquad |z|<1
$$
one defines 
$$
D(z)= \im(\Li_2(z))+\arg(1-z)\log|z|
$$
and checks that it extends to a real analytic function on
$\C\setminus \{0,1\}$, continuous on $\C$. See \cite{Za} for an account of
its many wonderful properties.  It is obvious that
\begin{equation}\label{tag30}
D(\bar z)=-D(z)\,.
\end{equation}
The 5-term relation satisfied by $D$ guarantees that, extended by
linearity to $\A(F)$, it induces a well defined function on $\B(\C)$
(still denoted by $D$). 

For $j=2$  \eqref{tag24} can be formulated as follows
$$
\begin{matrix}
l_2:&\B(F) & \longrightarrow & V\\
&\xi & \mapsto & \sum_{\sigma \in I} D(\xi^\sigma)\,
[\sigma] 
\end{matrix}
$$
(\eqref{tag30} makes it clear that the image $L_2$ lies in $V^-$) whose image
$L_2$ is a discrete subgroup of rank $n^-$.

An a priori lower bound for
$\left|\left|l_2(\xi)\right|\right|_1$ even for the simplest case
where $L_2$ is of rank $1$ (namely, for a field with only one complex
embedding) would be quite useful.  For example,
in \cite{BRV1} we find that an identity between the Mahler measure of
certain two-variable polynomials is equivalent to the following
\begin{equation}\label{tag32}
D(7[\alpha]+[\alpha^2]-3[\alpha^3]+[-\alpha^4])=0, \qquad
\alpha=(-3+\sqrt{-7})/4\,.
\end{equation}
This was proved by Zagier by showing that it is a consequence of
series of 5-term relations. Such calculations, however, can be quite
hard and at present there is no known algorithm that is guaranteed to
produce the desired result. Clearly if we knew a reasonable lower
bound for the possible non-zero values of $|D(\xi)|$ for $\xi \in
\B(\Q(\sqrt{-7}))$ a simple numerical verification would be enough to
prove \eqref{tag32}. 

Similarly, many identities \cite{BRV2} between the Mahler measure of
certain two-variable polynomials and $\zeta_F(2)$ for a corresponding
number field $F$, which by Borel's theorem are known up to an unspecified
rational number, could be proved by a numerical check. For example,
we can show that 
$$
m(x^2-2xy-2x+1-y+y^2)=s\frac{1728^{3/2}}{2^6\pi^7}\zeta_F(2)\,,
$$
with $s \in \Q^*$, where $F$ is the splitting field  $x^4-2x^3-2x+1$,
of discriminant $-1728$. However, though numerically $s$ appears to be
equal to $1$ we cannot prove this at the moment. Again, a reasonable
lower bound on $|D(\xi)|$ for non-torsion elements $\xi \in \B(F)$ would
allow us to conclude that $s=1$ by checking it numerically to high
enough precision. 

There is also some evidence that $\mu_{2,1}(F)$ might be universally
bounded below, at least for fields with one complex embedding. Indeed,
for a such a field one can construct a hyperbolic three dimensional
manifold $M$ by taking the quotient of hyperbolic space by a
torsion-free subgroup of the group of units of norm $1$ in a
quaternion algebra over $F$ ramified at all its real places. Its
associated Bloch group element $\xi(M)$, obtained from a triangulation
of $M$ into ideal tetrahedra, satisfies $D(\xi(M))=\Vol(M)$. On the
other hand, the volume of hyperbolic 3-manifolds is known to be
universally bounded below. The question becomes then, that of
obtaining an upper bound for the index in $\B(F)$ of the subgroup
generated by all such $\xi(M)$. This index is likely to be rather
small; in fact, if we accept a precise form of Lichtembaum's
conjecture, it should be essentially the order of $K_2(\mathcal{O}_F)$, an
analogue of a class group. Unfortunately, there is no known 
upper bound for $|K_2(\mathcal{O}_F)|$ in terms of, say, the degree and
discriminant of $F$.

Finally, to a hyperbolic 3-manifold $M$ with one cusp one may
associate \cite{CCGLS} a two variable polynomial $A(x,y) \in \Z[x,y]$,
called the {\it A-polynomial} of $M$. Its zero locus parameterizes
deformations of the complete hyperbolic structure of $M$.

It is known that
$$
m(A_\tau)=0
$$
for every face polynomial of $A$ and that $A$ is {\it reciprocal},
i.e. $A(1/x,1/y)=x^ay^bA(x,y)$ for some $a,b \in \Z$. It is
interesting that these two properties, which have
a topological and $K$-theoretic origin, are, for $A$ irreducible,
precisely the known necessary conditions for a polynomial in $\Z[x,y]$
to have  to have small Mahler measure (the first, an 
analogue of being the minimal polynomial of an algebraic unit, because
of \eqref{tag7}; the second because $m(P)$ is known to be universally bounded
below for $P$ non-reciprocal \cite{Sm1}).

 Though the whole picture is still not completely clear yet one 
can prove \cite{BRV2} for many $M$'s identities of the form
$$
2\pi m(A) = \left\|D(\xi(M)) \right\|_1\,,
$$
where $\xi(M)$ is the Bloch group element associated to $M$. This
suggests a direct link between Lehmer's conjecture and the size of the
invariants $\mu_{2,1}$.

\end{document}